\documentclass[11pt]{amsart}

\usepackage{upgreek}

\usepackage{ytableau}

\usepackage{braket}

\usepackage{amsmath, amsxtra, amsthm, amssymb,mathtools,bm,amsfonts}
\usepackage{mathabx}
\usepackage[normalem]{ulem}
\usepackage[mathscr]{euscript}
\usepackage{graphicx}
\usepackage{url}
\usepackage{color}
\usepackage{bbm}
\newcommand{\stkout}[1]{\ifmmode\text{\sout{\ensuremath{#1}}}\else\sout{#1}\fi}
\usepackage{tikz-cd}
\usepackage{tikz}
\usetikzlibrary{arrows.meta}

\usepackage[T1]{fontenc}
\graphicspath{/Figures/}

\usepackage{enumitem}
\usepackage{comment}
\usepackage[pdftex,hidelinks,backref=page]{hyperref}
\hypersetup{
    colorlinks,
    citecolor=magenta,
    filecolor=magenta,
    linkcolor=blue,
    urlcolor=black
}
\usepackage{cleveref}

\usepackage{xcolor}
\usepackage[colorinlistoftodos]{todonotes}
\renewcommand{\emptyset}{\varnothing}

\newcommand{\CC}{\mathbb C}
\newcommand{\PP}{\mathbb P}
\newcommand{\ZZ}{\mathbb Z}

\theoremstyle{definition}
\newtheorem{thm}{Theorem}[section]
\newtheorem{cor}[thm]{Corollary}
\newtheorem{lem}[thm]{Lemma}

\newtheorem{prop}[thm]{Proposition}
\newtheorem{defn}[thm]{Definition}

\newtheorem{eg}[thm]{Example}
\newtheorem{rem}[thm]{Remark}

\newtheorem{question}[thm]{Question}
\newtheorem{maintheorem}{Theorem}	
\newtheorem{fact}[thm]{Fact}

\numberwithin{equation}{section}
\allowdisplaybreaks

\newcommand{\ar}[1]
{{\xrightarrow{#1}}}

\newcommand{\indexedforests}{\operatorname{\mathsf{Forest}}}

\newcommand{\qsym}[2][]{
{\ifx&#1&%
  {\operatorname{QSym}_{#2}}
\else
  {{}^{#1}\!\operatorname{QSym}_{#2}}
\fi}
} 

\newcommand{\eqsym}[2][]{
{\ifx&#1&%
  {\operatorname{EQSym}_{#2}}
\else
  {{}^{#1}\!\operatorname{EQSym}_{#2}}
\fi}
} 

\newcommand{\qseq}[2][]{
{\ifx&#1&%
  {\operatorname{QSeq}_{#2}}
\else
  {{}^{#1}\!\operatorname{QSeq}_{#2}}
\fi}
}

\newcommand{\qsymide}[2][]{
{\ifx&#1&%
  {\operatorname{QSym}_{#2}^+}
\else
  {{}^{#1}\!\operatorname{QSym}_{#2}^+}
\fi}
} 

\newcommand{\eqsymide}[2][]{
{\ifx&#1&%
  {\operatorname{EQSym}_{#2}^+}
\else
  {{}^{#1}\!\operatorname{EQSym}_{#2}^+}
\fi}
} 

\newcommand{\sym}[1]{\operatorname{Sym}_{#1}} 
\newcommand{\compatible}[2][]{
{\ifx&#1&%
  {\mathcal{C}(#2)}
\else
  {\mathcal{C}^{m}(#2)}
\fi}
} 
\newcommand{\internal}[1]{\operatorname{IN}(#1)} 
\newcommand{\suchthat}{\;|\;}

\newcommand{\grass}[1]{\operatorname{Grass}_{#1}} 
\newcommand{\qgrass}[1]{\operatorname{QGrass}_{#1}} 

\newcommand{\des}[1]{\operatorname{Des}(#1)} 
\newcommand{\desnc}[1]{\operatorname{Des}_{\operatorname{NC}}(#1)} 

\date{}

\newcommand{\idem}{\operatorname{id}} 
\newcommand{\rope}[1]{\mathsf{R}_{#1}} 

\newcommand{\fl}[1]{\mathrm{Fl}_{#1}}
\newcommand{\GL}{\operatorname{GL}}

\definecolor{ao}{rgb}{0.0, 0.5, 0.0}

\newcommand{\hhmp}{\mathrm{QFl}}

\newcommand{\wt}[1]{\widetilde{#1}} 

\newcommand{\tl}{\textbf{t}}
\newcommand{\xl}{\textbf{x}}

\newcommand{\NC}{\operatorname{NC}}
\newcommand{\ForToNC}{\operatorname{ForToNC}}

\newcommand{\inv}[1]{\operatorname{Inv}(#1)}
\newcommand{\invnc}[1]{\operatorname{Inv}_{\NC}(#1)} 

\renewcommand\emph[1]{\textcolor{blue}{\textit{#1}}} 

\newcommand{\OHL}
{\operatorname{OHL}}

\newcommand{\pluckervanishing}[1]{\operatorname{PV}_{#1}}

\newcommand{\invq}[1]{\operatorname{Inv}^Q(#1)}
\newcommand{\Gr}{\operatorname{Gr}}
\newcommand{\QGr}
{\operatorname{QGr}}
\newcommand{\Zigzag}
{\operatorname{Zigzag}}

\newcommand{\Comp}
{\operatorname{Comp}}
\newcommand{\Part}
{\operatorname{Part}}

\definecolor{col1}{RGB}{210,90,90}
\definecolor{col2}{RGB}{90,190,90}
\definecolor{col3}{RGB}{90,110,210}

\title{The Quasisymmetric Grassmannian}

\author{Nantel Bergeron}
\address{Dept. of Math. and Stat., York University, Toronto, ON M3J 1P3, Canada}
\email{\href{mailto:bergeron@yorku.ca}{bergeron@yorku.ca}}

\author{Lucas Gagnon}
\address{Department of Mathematics,  University of Southern California, Los Angeles, CA 90089, USA}
\email{\href{mailto:lgagnon@usc.edu}{lgagnon@usc.edu}}

\author{Hunter Spink}
\address{Department of Mathematics,
University of Toronto, Toronto, ON M5S 2E4, Canada}
\email{\href{mailto:hunter.spink@utoronto.ca}{hunter.spink@utoronto.ca}}

\author{Vasu Tewari}
\address{Department of Mathematical and Computational Sciences, University of Toronto Mississauga, Mississauga, ON L5L 1C6, Canada}
\email{\href{mailto:vasu.tewari@utoronto.ca}{vasu.tewari@utoronto.ca}}
\thanks{
NB was supported by the Natural Sciences and Engineering Research Council of Canada (NSERC) and York Research Chair in Applied Algebra.
HS and VT acknowledge the support from NSERC, respectively [RGPIN-2024-04181] and [RGPIN-2024-05433].}


\begin{document}

\begin{abstract}
We construct a complex of toric varieties we call the quasisymmetric Grassmannian inside the Grassmannian of $r$-planes in $\mathbb{C}^n$. Each irreducible component is a positroid variety and an $S_n$ translate of a toric Richardson variety of ribbon shape. We describe it as the vanishing locus of equations  $\Delta_A\Delta_{A'}=0$ in  Pl\"ucker coordinates determined by a new noncrossing combinatorial object we call the quasisymmetric Johnson graph. We give an affine paving, and show that its cohomology ring is a quasisymmetric modification of the Borel presentation of the Grassmannian's cohomology, with fundamental quasisymmetric polynomials playing the role of Schur polynomials.
\end{abstract} 

\maketitle

\section{Introduction}

There is a rich connection between symmetric polynomials and the Grassmannian manifold $\Gr(r; n)$ of $r$-planes in $\mathbb{C}^n$. 
Let $\mathrm{Part}_{r, n}$ be the set of integer partitions contained in an $r \times (n - r)$ rectangle. 
Every $\Gr(r; n)$ is stratified by Schubert cells $\{\mathring{X}^{\lambda} \suchthat \lambda \in \mathrm{Part}_{r, n}\}$, and the Kronecker dual of the homology basis of Schubert cycles $X^{\lambda} \coloneqq \overline{\mathring{X}^{\lambda} }$ is the cohomology basis of Schur symmetric polynomials $s_\lambda(x_1,\ldots,x_r)\in H^\bullet(\Gr(r;n))$ in the negative Chern roots $x_{1}, \ldots, x_{r}$ of the tautological subbundle.
This fact is intimately tied to two equivalent and classical cohomology presentations: 
\begin{align*}
H^{\bullet}(\Gr(r; n)) 
\cong \frac{\mathrm{Sym}_{r}}{\langle s_\lambda(x_1,\ldots,x_r)\suchthat \lambda \not\in \Part_{r,n}\rangle} 
\cong \frac{\mathrm{Sym}_{r} \otimes \mathrm{Sym}_{n-r}}{\langle f - f(0) \suchthat f \in \mathrm{Sym}_{n} \rangle}
\end{align*}
where $\mathrm{Sym}_{k}$ is the ring of symmetric polynomials in $k$ variables. In fact, the ideal in the first presentation has a free $\mathbb{Z}$-basis $\{s_{\lambda}(x_1,\ldots,x_r) \suchthat \ell(\lambda)\le r,\, \lambda \notin  \mathrm{Part}_{r, n}\}$.

Our aim is to give a parallel construction adapted to the {quasisymmetric polynomials}, the \emph{quasisymmetric Grassmannian}. This will be an equidimensional algebraic variety $\QGr(r;n)\subset \Gr(r;n)$ composed of $\binom{n-2}{r-1}$ many $(n-1)$-dimensional torus-orbit closures.

For an integer composition $\alpha = (\alpha_{1}, \ldots, \alpha_{\ell})$ of $m$, the \emph{fundamental quasisymmetric polynomial} is
\[
F_{\alpha}(x_{1}, \ldots, x_{r}) = \hspace{-1.5em} \sum_{\substack{ 1 \le i_{1} \le \cdots \le i_{m} \le r \\ \text{$i_{k} < i_{k+1}$ if $k \in \mathrm{Des}(\alpha)$}}} \hspace{-1.5em} x_{i_{1}} x_{i_{2}} \cdots x_{i_{m}}
\qquad\text{where $\textstyle \mathrm{Des}(\alpha) = \{\sum_{i = 1}^{k} \alpha_{i} \suchthat 1 \le k < \ell \}$}.
\]
Stanley~\cite{StThesis} identified a natural decomposition of Schur polynomials into fundamental quasisymmetric polynomials, and Gessel~\cite{Ges84} observed that the $\{F_{\alpha}(x_1,\ldots,x_r)\suchthat \ell(\alpha)\le r\}$ span a ring containing $\mathrm{Sym}_{r}$, the \emph{ring of quasisymmetric polynomials $\qsym{r}$}.  In order to state the main result, let $\mathrm{Comp}_{r, n}$ be the set of compositions whose ribbon diagram is contained in an $r \times (n-r)$ rectangle.  

\begin{maintheorem}[Theorems~\ref{thm:QuotientPresentation} and~\ref{thm:Qsymtensor}]
\label{introthm1}
The variety $\QGr(r; n) \subseteq \Gr(r; n)$ has
\[
H^{\bullet}(\QGr(r; n)) 
\cong \frac{\mathrm{QSym}_{r}}{\langle F_\alpha(x_1,\ldots,x_r)\suchthat \alpha \not\in \Comp_{r,n}\rangle} 
\cong \frac{\mathrm{QSym}_{r} \otimes \mathrm{QSym}_{n-r}}{\langle f - f(0) \suchthat f \in \mathrm{QSym}_{n} \rangle}.
\]
In fact, the ideal in the first presentation has a free $\mathbb{Z}$-basis $\{F_{\alpha}(x_1,\ldots,x_r) \suchthat \ell(\alpha)\le r,\, \alpha \notin  \mathrm{Comp}_{r, n}\}$.
\end{maintheorem}
  
The definition of $\QGr(r; n)$ follows naturally from the authors' work with P. Nadeau on the \emph{quasisymmetric flag variety} $\mathrm{QFl}_{n}$~\cite{BGNST2}, an equidimensional toric complex inside the complete flag variety $\mathrm{Fl}_{n}$. The cohomology compares with that of the flag variety as follows.
\begin{align*}
\text{Borel's theorem \cite{Bor53} states:} && H^{\bullet}(\mathrm{Fl}_{n}) &\cong \ZZ[x_1,\ldots,x_n] \big/ \langle f - f(0) \;|\; f \in \mathrm{Sym}_{n} \rangle,\qquad\text{and}\qquad \\[1ex]
\text{\cite[Theorem A]{BGNST2} states:} && H^{\bullet}(\mathrm{QFl}_{n}) &\cong \ZZ[x_1,\ldots,x_n] \big/ \langle f - f(0) \;|\; f \in \mathrm{QSym}_{n} \rangle.
\end{align*}
 This geometrically realizes the ``quasisymmetric coinvariants'' first studied in~\cite{ABB04}. In contrast, the quasisymmetric Grassmannian gives a direct cohomological interpretation of  $\mathrm{QSym}_{r}$.
\begin{defn}
\label{introdef_qgr}
The Quasisymmetric Grassmannian $\QGr(r; n)$ is the projection $\pi\big( \mathrm{QFl}_{n} \big) \subseteq \Gr(r; n)$ under the natural projection map $\pi: \mathrm{Fl}_{n} \to \Gr(r; n)$.
\end{defn}

Morally, one can view elements of $H^{\bullet}(\QGr(r; n))$ as \textit{quasisymmetric polynomials} in the Chern roots of the tautological subbundle on $\QGr(r; n)$. As the containment $\sym{r}\subset \qsym{r}$ is proper for $r\ge 2$, these classes often do not come from the pullback $H^\bullet(\Gr(r;n))\to H^\bullet(\QGr(r; n))$. 

Quasisymmetry is a specialization of a stronger condition called \emph{equivariant quasisymmetry} on polynomials in two sets of variables $f(x_1,\ldots,x_r;t_1,\ldots,t_n)$ introduced by us and P. Nadeau \cite{BGNST1}. 
The incidence structure of $T$-invariant curves in $\QGr(r;n)$ imposes constraints on its $T$-equivariant cohomology via GKM theory \cite{GKM98}, which corresponds to equivariant quasisymmetry.

Our construction substantially differs from the infinite-dimensional James space $\Omega\Sigma\mathbb{P}^{\infty}$ of Baker--Richter~\cite{BaRi08}, Oesinghaus~\cite{Oe19}, and Pechenik--Satriano~\cite{PeSa22,PeSa24}, which realizes its cohomology ring as quasisymmetric power series.
The primary difference is that the James space construction has a natural cohomological basis corresponding to the \emph{monomial quasisymmetric functions} $M_\alpha=\sum_{i_1<\cdots<i_{\ell}}x_{i_1}^{\alpha_1}\cdots x_{i_{\ell}}^{\alpha_\ell}$. 
Our construction of $\QGr(r;n)\subset \Gr(r;n)$ allows us to leverage the relationship between Schur and fundamental quasisymmetric polynomials; see Section~\ref{intosubsection3}.

\begin{question}
Is there a geometric relationship between the constructions $\QGr(r; n)$ and  $\Omega\Sigma\mathbb{P}^{\infty}$?
\end{question}

\subsection{Structure and characterizations}

Beyond Definition~\ref{introdef_qgr}, we give two intrinsic characterizations of $\QGr(r; n)$.  
Each exploits the fact that the incidences of $T$-invariant curves in $\Gr(r; n)$ are encoded by the \emph{Johnson graph} $J_{r, n}$ on $r$-element subsets of $\{1,\ldots,n\}$, with edges given by single-element swaps.  
In Section~\ref{sec:combinatorics_qgr} we define a spanning subgraph, the \emph{quasisymmetric Johnson graph} $QJ_{r, n}$ encoding the incidences of curves in $\QGr(r;n)$ by removing all edges that violate a certain non-crossing condition; see Figure~\ref{fig:QJ} or the $1$-skeleton of Figure~\ref{fig:placeholder} for an example.  

Our first characterization of $\QGr(r; n)$ makes use of the $\binom{n}{r}$-many Pl\"{u}cker coordinates on $\Gr(r; n)$, which we write as $\Delta_{A}$ for each vertex $A$ of the Johnson graph.  

\begin{maintheorem}[Theorem~\ref{thm:everythingequal}]
\label{introthm2}
As a subvariety of $\Gr(r; n)$, $\QGr(r; n)$ is the common vanishing locus of the Pl\"ucker coordinate products $\{\Delta_{A}\Delta_{B}\suchthat AB \in E(J_{r, n}) \setminus E(QJ_{r, n})\}$.
\end{maintheorem}

For the second characterization, we write $\lambda_{A} \in \Part_{r, n}$ for the partition canonically associated to $A \in V(J_{r, n})$, as recalled in Section~\ref{sec:combinatorics_gr}.  
If we orient $J_{r, n}$ according to inclusion of partitions, then the incoming edges to $A$ correspond precisely to the boxes in $\lambda_{A}$, which index the coordinates of $\mathring{X}^{\lambda_A}\cong \mathbb{A}^{|\lambda_A|}$.  
We then define the \emph{quasisymmetric Grassmannian Schubert cell} $\mathring{X}_{\NC}^{\lambda}$ to be the coordinate subspace of $\mathring{X}^\lambda$ corresponding to the incoming edges from $QJ_{r, n}$. The dimension of $\mathring{X}_{\NC}^{\lambda}$ is the \emph{outer hook length} $\OHL(\lambda)=\lambda_1+\ell(\lambda)-1$, the length of the longest hook.

\begin{maintheorem}[Theorem~\ref{thm:everythingequal}]
\label{introthm3}
There is an affine paving of $\QGr(r;n)$ by quasisymmetric Grassmannian Schubert cells, 
\[
\QGr(r; n) = \bigsqcup_{\lambda \in \mathrm{Part}_{r, n}} \mathring{X}_{\NC}^{\lambda}
\qquad\text{with}\qquad
\mathring{X}_{\NC}^{\lambda} \cong \mathbb{A}^{\OHL(\lambda)}.
\]
Consequently, the odd Betti numbers of $\QGr(r; n)$ vanish and the even Betti numbers are
\[
\dim H_{2k}(\QGr(r;n))=\sum_{\substack{a \le r,\, b\le n-r, \\ \text{and $a+b = k+1$}}} \binom{k-1}{a-1}.
\]
\end{maintheorem}

From this result we see that $\QGr(r; n)$ is the union of the cell closures 
$X_{\NC}^{\lambda} \coloneqq \overline{\mathring{X}_{\NC}^{\lambda}}$
which we call \emph{quasisymmetric Grassmannian Schubert cycles}.  We find that they are \emph{positroid varieties} in the sense of Knutson--Lam--Speyer \cite{KLS13}, the complex closures of Postnikov's positroid cells that stratify the totally nonnegative Grassmannian \cite{Po06}, and use this to show that the complex of moment polytopes faithfully captures the combinatorics of $\QGr(r;n)$ as a complex of toric varieties. 

\begin{figure}[!ht]
    \centering
    \includegraphics[scale=0.8]{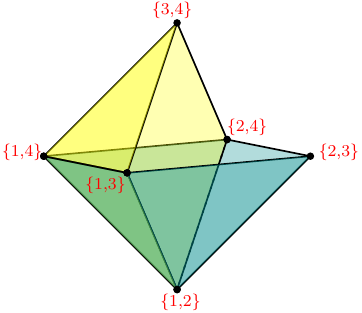}
    \caption{$\QGr(2;4)$ is the union of two isomorphic $3$-dimensional singular projective toric varieties whose moment polytopes are square pyramids. The intersection corresponds to their common faces, two triangles which share an edge, which is geometrically two $\mathbb{P}^2$s sharing a $\mathbb{P}^1$. The $1$-skeleton of this complex is the quasisymmetric Johnson graph $QJ_{2,4}$, and  $\QGr(2;4)=\{\Delta_{34}\Delta_{23}=0\}\subset \Gr(2;4)$.}
    \label{fig:placeholder}
\end{figure}

\subsection{A geometric meaning for $F_{\alpha}$}
\label{intosubsection3}

We conclude with a re-interpretation of the transition from Schur polynomials to fundamental quasisymmetric polynomials. 
By Theorem~\ref{introthm3}, the cell closures $X_{\NC}^{\lambda}$ give a homology basis for $\QGr(r; n)$.  
While $\QGr(r; n)$ does not enjoy Poincar\'{e} duality, we find a \textit{Kronecker} dual of this basis under the natural pairing $H^{k}(\QGr(r; n)) \otimes H_{k}(\QGr(r; n)) \to \ZZ$, exactly paralleling the Kronecker duality between the bases of $s_\lambda(x_1,\ldots,x_r)$ and $X^\lambda$ for $\Gr(r;n)$.

In Section~\ref{sec:combinatorics_qgr} we introduce an apparently novel bijection $\alpha \mapsto \lambda_{\alpha}$ from compositions to partitions. 

\begin{maintheorem}[Theorem~\ref{thm:Kronecker}]
\label{introthm4}
The cohomological basis $\{F_{\alpha}(x_1,\ldots,x_r) \suchthat \alpha \in \mathrm{Comp}_{r, n}\}\subset H^\bullet(\QGr(r;n))$ is Kronecker dual to the homology basis $\{[X_{\NC}^{\lambda} ] \suchthat \lambda \in \mathrm{Part}_{r, n}\}\subset H_\bullet(\QGr(r;n))$, or equivalently
\[
\int_{X_{\NC}^{\lambda_\beta}} F_{\alpha}(x_1,\ldots,x_r) = \langle [X_{\NC}^{\lambda_\beta}],F_\alpha(x_1,\ldots,x_r)\rangle_{\QGr(r;n)}= \delta_{\alpha,\beta}.
\]
\end{maintheorem} 

This statement allows us to generalize a classical result of Gessel~\cite{Ges84}.  
If we interpret a polynomial $f(x_{1}, \ldots, x_{r}) \in \mathrm{QSym}_{r}$ as an element of $H^\bullet(\QGr(r;n))$, then the Kronecker duality implies
\[
\int_{X_{\NC}^{\lambda_\alpha}}f = [F_{\alpha}](f)
\qquad\text{where}\qquad
[F_{\alpha}]f \coloneq \text{the coefficient of $F_{\alpha}$ in $f$}.
\]
For $f\in \sym{r}$, Gessel \cite{Ges84} shows (in combinatorial language) that the same equation holds if we replace $X^\lambda_{\NC}$ by a Richardson variety $X^\eta_{\nu}\subset \Gr(r;n)$ with $\eta/\nu$ of ribbon shape $\alpha$, so
$$[X^\lambda_{\NC}]=\sum_{\mu}([F_\alpha]s_\mu)[X^\mu]=[X^{\eta}_{\nu}]\in H_\bullet(\Gr(r;n)).$$
This coincidence is satisfactorily explained by the fact that $X^\lambda_{\NC}$ is in fact an $S_n$-translate of $X^\eta_\nu$, a computation that was essentially carried out in \cite{nst_c} which we shall recall in Section~\ref{sec:Positroid}.

\subsection{Outline} 
Section~\ref{sec:combinatorics_gr} recalls the classic combinatorics of $\Gr(r; n)$ and Section~\ref{sec:combinatorics_qgr} introduces new ``noncrossing'' generalizations thereof. 
Sections~\ref{sec:PlVan} and~\ref{sec:QGr} give alternative characterizations of $\QGr(r;n)$ by affine paving and Pl\"ucker equations.
Section~\ref{sec:Positroid} relates the paving strata to positroid and translated Richardson varieties before establishing a rigidity theorem for $\QGr(r;n)$.
Section~\ref{sec:EQSym} recalls properties of quasisymmetric polynomials for
Section~\ref{sec:Cohomology}, which proves all cohomological claims. 
Section~\ref{sec:morefacts} establishes the equidimensionality of $\QGr(r;n)$ and describes its moment polytopes.

\section{Combinatorics of the Grassmannian}
\label{sec:combinatorics_gr}

This section recalls the essential combinatorics of Grassmannian Schubert calculus, including $r$-subsets, partitions, and Grassmannian permutations ~\cite{AF24, BjBr05,Mac95, St12, EC2}. 
We let $n$ be a nonnegative integer, $[n]\coloneqq \{1,\ldots,n\}$, and $S_{n}$ be the symmetric group on $n$ letters, generated by the simple transpositions $s_i=(i,i+1)$ for $1\le i \le n-1$.

Our primary combinatorial objects will be $r$-element subsets $A \subseteq [n]$, and we denote by $\binom{[n]}{r}$ the collection of all such subsets.  
We sometimes represent $A \in \binom{[n]}{r}$ using the sets
\[
L \coloneqq [r]\setminus A = \{a_k < \cdots < a_1\}
\qquad\text{and}\qquad
R \coloneqq A \setminus [r] = \{b_1 < \cdots < b_k\},
\]
which necessarily have the same size, and whose union is the symmetric difference of $A$ and $[r]$.  We will use the notation $A = (L \mid_r R)$, and from this representation we recover $A = ([r]\setminus L)\sqcup R$.

The \emph{Johnson graph $J_{r,n}$} is the graph on vertex set $\binom{[n]}{r}$ with edges $AB$ whenever $B=(A\setminus j)\cup i$ for $i\neq j$.  
See Figure~\ref{fig:grassmannian_objects_bij} for an example. 
Section~\ref{sec:RecollGrass} explains how $J_{r,n}$ encodes the geometry of the Grassmannian $\Gr(r; n)$ by identifying vertices and edges with torus-invariant points and curves. 

The \emph{inversion set} of $A \in \binom{[n]}{r}$ is
\[
\inv{A}\coloneqq\{(i,j)\suchthat i<j,\text{ }j\in A\text{ and }i\not\in A\}.
\]
The \emph{Gale order} $\le$ is the partial order on $\binom{[n]}{r}$ generated by $A>B$ if $B = (A \setminus j) \cup i$ for $(i, j) \in \inv{A}$. In this ordering the minimum is $[r]$, the maximum is $\{n-r+1,\ldots,n\}$, and
\[
A \le B
\qquad \text{if} \qquad
|A \cap [k]| \ge |B \cap [k]| \quad \text{for all } 1 \le k \le n.
\]

\begin{lem}
\label{lem:invA}
Let $A=(a_k,\ldots,a_1\mid_r b_1,\ldots,b_k)$ and set $b_0\coloneqq r$. 
Then $\inv{A}$ consists of:
\begin{enumerate}[label=(\arabic*)]
    \item\label{invA1} $(i,b_p)$ where $i\notin A$ and $i<b_p$, for $1\leq p\leq k$;
    \item\label{invA2} $(a_p,j)$ where $j\in A\cap[r]$ and $a_p<j$, for $1\leq p\leq k$.
\end{enumerate}
\end{lem}
\begin{proof}
Cases (1) and (2) describe those $(i,j)\in \inv{A}$ with $r<j$ and $r\ge j$ respectively.
\end{proof}

\begin{figure}[!ht]
    \includegraphics[scale=0.8]{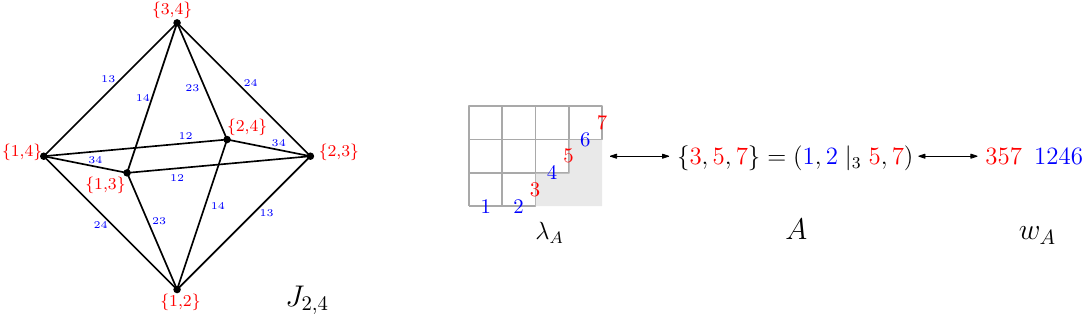}
    \caption{The Johnson graph $J_{2,4}$ (left) and an example of the bijections $\Part_{r,n}\leftrightarrow \binom{[n]}{r}\leftrightarrow \grass{r,n}$ (right) for $r=3$, $n=7$.}
    \label{fig:grassmannian_objects_bij}
\end{figure}

In addition to $r$-subsets of $[n]$, partitions and Grassmannian permutations are also used in the study of the Grassmannian.  
We now recall the bijections between these equivalent combinatorial structures; see Figure~\ref{fig:grassmannian_objects_bij} for an example.  
As in the introduction, let $\Part_{r,n}$ be the set of partitions $\lambda = (\lambda_{1} \ge \lambda_{2} \ge \cdots \ge \lambda_{r} \ge 0)$ with $\lambda_{1} \le n-r$.  
There is a bijection
\begin{equation}
\label{eq:sets_to_parts}
\begin{array}{rrcl}
& \binom{[n]}{r} &\to& \Part_{r,n} \\
& \{i_1< i_{2} <\cdots < i_r \} = A & \mapsto & \lambda_{A} \coloneq \big(i_r-r,\ldots, i_{2} - 2, i_1-1\big).
\end{array}
\end{equation}

Under the bijection $A \leftrightarrow \lambda_A$, Gale order corresponds to inclusion order, and each box of $\lambda_A$ corresponds to an inversion of $A$. 
Reading the boundary of $\lambda$ from bottom left to top right as shown in Figure~\ref{fig:grassmannian_objects_bij}, $A$ indexes the north steps, and each $(i, j) \in \inv{A}$ determines a box of $\lambda_{A}$ in the row of north step $j$ and column of east step $i$.  

We can also characterize edges in the Johnson graph using \emph{rim hooks}~\cite[Chapter 7.17]{EC2}, which are connected skew shapes containing no $2\times 2$ square. 
For $(i, j) \in \inv{A}$, the partition $\lambda_{(A\setminus j)\cup i}$ is obtained from $\lambda_{A}$ by removing  the rim hook of boundary boxes between steps $i$ and $j$.  
 
\begin{figure}[!ht]
    \includegraphics[scale=0.8]{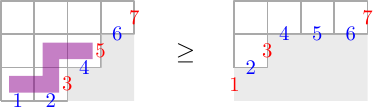}
    \caption{Rim hook removal corresponding to inversion $(1,5)$ from $\lambda_A$ where $A=\{3,5,7\}$}
    \label{fig:rim_hook_removal}
\end{figure}

The \emph{Frobenius symbol}~\cite[Chapter I.1]{Mac95} of $\lambda$ records arm and leg sizes of each box $b$ on the main diagonal of $\lambda$ (not including $b$ itself in the sizes). 
If $A = (L |_{r} R)$, the Frobenius symbol of $\lambda_A$ is 
\begin{equation}
\label{eqn:FrobSymbol}
\begin{pmatrix}
\mathrm{Arm}_1 & \cdots & \mathrm{Arm}_k \\
\mathrm{Leg}_1  & \cdots & \mathrm{Leg}_k
\end{pmatrix}=\begin{pmatrix}
b_k - r-1 & \cdots & b_1 - r-1 \\
r - a_k  & \cdots & r - a_1
\end{pmatrix}.
\end{equation}

We now describe the combinatorial equivalence between $\binom{[n]}{r}$ and  Grassmannian permutations.  
For $w \in S_{n}$, the \emph{(left) inversion set} and \emph{length} of $w$ are
\[
\inv{w}\coloneqq\{(i,j)\suchthat i<j,\ w^{-1}(i)>w^{-1}(j)\}
\qquad\text{and}\qquad
\ell(w)\coloneqq|\inv{w}|.
\]
The \emph{Bruhat order} $u\le v$ on $S_n$ is generated by $u> (i,j)u$ for $(i,j)\in \inv{u}$.
The \emph{descent set} of $w$ is $\des{w}\coloneqq\{i\suchthat w(i)>w(i+1)\}$. 
If $k\in \des{w}$ then $w s_k<w$ and $ws_kw^{-1}\in \inv{w}$.  
Let $\grass{r,n}$ be the set of \emph{$r$-Grassmannian permutations} $\{w \in S_{n}\suchthat\des{w}\subseteq\{r\}\}$.  
There is a bijection
\begin{equation}
\begin{array}{rcl}
\binom{[n]}{r} & \to & \grass{r,n}\\
A & \mapsto & w_{A}
\end{array}
\qquad\text{defined by}\qquad
\begin{array}{rl}
A=&\!\!\!\!\{w_A(1)<\cdots<w_A(r)\},\,\text{and} \\[0pt]
[n] \setminus A=&\!\!\!\!\{w_A(r+1)<\cdots<w_A(n)\}
\end{array}
\end{equation}
There is a natural action of $S_n$ on $\binom{[n]}{r}$. By construction $w_A$ is the unique minimal length permutation with $w_A\cdot[r]=A$ and each
inversion $(i,j)\in\inv{w_A}$ satisfies $w_A^{-1}(i)\geq r>w_A^{-1}(j)$.  
By construction $\inv{w_A}=\inv{A}$ and the inversions of $w_A$ correspond to boxes of $\lambda_A$ as above.

\section{Combinatorics of the Quasisymmetric Grassmannian}
\label{sec:combinatorics_qgr}

This section introduces the combinatorial objects essential for the quasisymmetric Grassmannian including noncrossing partitions, the quasisymmetric Johnson graph, quasigrassmannian permutations, and compositions, analogous to the constructions recalled in  Section~\ref{sec:combinatorics_gr}.

\subsection{Noncrossing partitions and the quasisymmetric Johnson graph}

A \emph{set partition} of $[n]$ is a collection of non-empty disjoint subsets 
$B_1, \ldots, B_k$, called \emph{blocks}, whose union is $[n]$. A 
\emph{crossing} is a 4-tuple $a < b < c < d$ in $[n]$ such that $a, c$ belong 
to one block and $b, d$ belong to another. A set partition is \emph{noncrossing} 
if it has no crossing. 

Coxeter-Catalan combinatorics, when specialized to the product of simple transpositions for the standard Coxeter element $s_{n-1}\cdots s_1=(n\,n-1\,\ldots\,1)$, transforms the noncrossing set partitions into a collection $\NC_{n} \subseteq S_{n}$ as follows. 
From a noncrossing set partition of $[n]$, each block $B = \{b_1 < \cdots < b_s\}$ determines a  backwards cycle $(b_s, b_{s-1}, \ldots, b_1)$. 
The product of these (disjoint) cycles over 
all blocks is an \emph{(algebraic) noncrossing partition} and $\NC_n$ is the set of all algebraic noncrossing partitions. 
Given this bijection we shall use the terms blocks and cycles interchangeably.
For example, the noncrossing set partition with blocks $\{\{1,5,6\},\{2,3\}, \{4\}\}$ of $[6]$ gives the algebraic noncrossing partition $(6\,5\,1)(3\,2) = 632415 \in \NC_6$.

The \emph{Kreweras order}~\cite{Kre72} is the partial order on $\NC_n$ induced by refinement of set partitions, with $\idem$ at the bottom and $(n\,n-1\,\ldots\,1)$ at the top. 
We are primarily concerned with the Hasse diagram of this order, so we describe its covering relations: $u, w \in \NC_{n}$ have a covering relation if there is a transposition $\tau = (i\,j)$ such that $\tau u=w$.  
See Figure~\ref{fig:QJ} for an example of the Kreweras order on $\NC_{4}$.
\begin{defn}
The \emph{quasisymmetric Johnson graph} $QJ_{r,n}$ is the spanning subgraph of $J_{r, n}$ with edges from vertices $A$ to $B$ if and only if under the natural action of $S_n$ on $\binom{[n]}{r}$ we have
\[
A = w \cdot [r]
\qquad\text{and}\qquad
B = u \cdot [r]
\qquad\text{for some $w, u \in \NC_{n}$ Kreweras-adjacent.}
\]
See Figure~\ref{fig:QJ} for an example of $QJ_{4, 2}$. 
\end{defn}

We define the \emph{quasisymmetric inversion set} of $A\in\binom{[n]}{r}$ to be
\[
\invq{A}\coloneqq \{(i,j)\suchthat (i,j)\in \inv{A}\text{ and }AB\in QJ_{r,n}\text{ for }B=(A\setminus j)\cup i\}.
\]
The structure of $QJ_{r, n}$ and each $\invq{A}$ is controlled by the fibers
\[
\NC_n^A\coloneqq \{w\in \NC_n\suchthat w\cdot [r]=A\},
\]
each of which is nonempty: if $A=(a_k,\dots,a_1\mid_r b_1,\dots,b_k)\in \binom{[n]}{r}$ then the permutation $w = \prod_{1\leq i\leq k} (a_i,b_i)$ is noncrossing and also has $w \cdot [r]=A$, so $w \in \NC_n^A$.

\begin{figure}[!ht]
    \includegraphics[width=\linewidth]{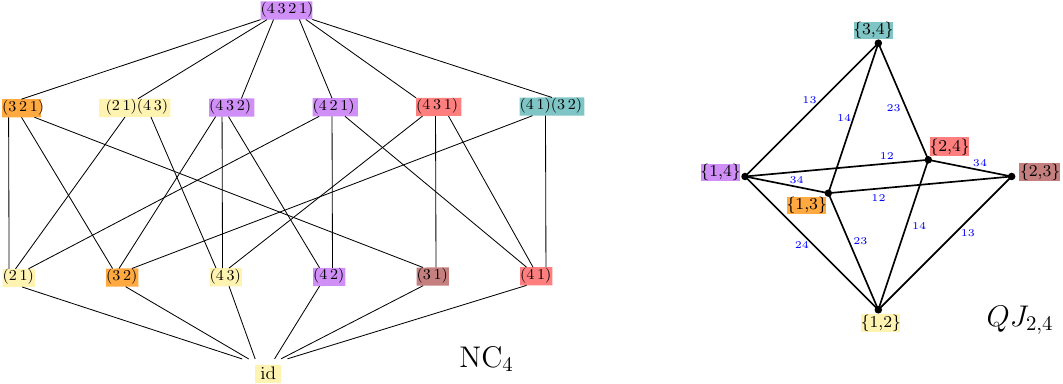}
\caption{The Kreweras lattice $\NC_4$ (left) written using cycle notation and the Quasisymmetric Johnson graph $QJ_{2,4}$ (right), with color denoting the fibers $\NC_{4}^{A}$.
}
\label{fig:QJ}
\end{figure}

\subsection{Quasigrassmannian permutations, $\Comp_{r,n}$, and $\Part_{r,n}$}
\label{sec:QJn}

For $w \in \NC_{n}$, the \emph{noncrossing descent set} of $u$ is defined to be
\[
\desnc{w} \coloneqq \{ k \in \des{w} \suchthat ws_k \in \NC_{n}\}.
\]
\begin{fact}[{\cite[Lemma 7.4]{BGNST1}}]
    \label{fact:ncdescent}
    For $w\in \NC_n$ and $i\in \des{w}$, we have $i\in \desnc{w}\leftrightarrow i\in \{w(i),w(i+1)\}$.
\end{fact}

We now introduce the noncrossing analogue of Grassmannian permutations.

\begin{defn}
A permutation $w \in \NC_{n}$ is \emph{$r$-Quasigrassmannian} if $\desnc{w}\subseteq\{r\}$.  Denote the set of Quasigrassmannian permutations by $\qgrass{r,n}\subset \NC_n$.
\end{defn}

Despite the similar definitions, we emphasize here that Quasigrassmannian permutations are \textit{not} in general Grassmannian; for example $w=4321=(41)(32)\in \qgrass{2,4}$ is not Grassmannian.

\begin{lem}
    With $b_0\coloneqq r$, the map
    \begin{equation}
        \label{eq:sets_to_qgrass}
        \begin{array}{rcl}
            \binom{[n]}{r} &\to& \qgrass{r,n} \\
            A=(a_k,\dots,a_1 \mid_r b_1,\dots,b_k) & \mapsto & z_A\coloneqq \prod_{i=1}^k\big(b_i\,(b_i-1)\,\ldots\, (b_{i-1}+1)\,a_i\big),
        \end{array}
    \end{equation}
    is a bijection. We will henceforth call $z_A$ the \emph{quasigrassmannian permutation associated to $A$}.
\end{lem}
\begin{proof}
Either $z_{A} = \idem$ or $\des{z_{A}} = \{a_{1}, \ldots, a_{k}, r, b_{k-1}, \ldots, b_{1}\}$.  
Of these descents, only $r$ meets the criterion of Fact~\ref{fact:ncdescent}, so $z_{A} \in \qgrass{r,n}$.  On the other hand, $z_{A} \cdot [r] = A$, giving an inverse.  
\end{proof}

For example, if $n=10$, $r=4$, and $A=\{1,3,7,9\}=(2,4|_4 7,9)$ then $z_A=(9\,8\,2)(7\,6\,5\,4)$.

\begin{lem}
\label{le:welldefined}
For $A\in \binom{[n]}{r}$, $z_A$ belongs to the fiber $\NC_n^A$, and every noncrossing partition in $\NC_{n}^A$ can be reduced to $z_A$ by a sequence of transformations $u \mapsto u s_k$ with $k \in \desnc{u}$.
In particular, $z_A$ is the unique Bruhat-minimal element of $\NC_{n}^A$.
\end{lem}
\begin{proof}
The first claim is immediate.  The second follows from induction on length, as every element of $\NC_{n} \setminus \qgrass{r,n}$ has a noncrossing descent in $[n] \setminus \{r\}$.
\end{proof}

\begin{rem}
    It is shown in~\cite[Corollary 7.9]{BGNST1} that $\NC_n$ is the equivalence class of the identity permutation under the equivalence  relation on $S_n$ where $w \sim w s_i$ whenever $i\in \{w(i),w(i+1)\}$.
    Lemma~\ref{le:welldefined} and Fact~\ref{fact:ncdescent} then say that if we restrict this relation to $i\neq r$ then the equivalences classes are the fibers $\NC_{n}^A$.
\end{rem}

As in the introduction, let $\Comp_{r,n}$ denote the set of compositions $\alpha = (\alpha_{1}, \ldots, \alpha_{k})$ whose ribbon diagrams fit in an $r\times (n-r)$ box, i.e. with $k \le r$ and $\sum \alpha_{i} - k+1 \le n-r$.  

\begin{lem}
    For $z_A\in \qgrass{r,n}$, let $B_i$ be the block with $\min(B_i)=i$ for $1\leq i\leq r$.
    The map
    \begin{equation}
        \label{eq:qgrass_to_comp}
        \begin{array}{rcl}
            \qgrass{r,n} &\to& \Comp_{r,n} \\
            z_A & \mapsto & \alpha_A\coloneqq (|B_{a_k}|-1,|B_{(a_k)+1}|,\dots,|B_{r-1}|,|B_r|)
        \end{array}
    \end{equation}
    is a bijection.
\end{lem}
\begin{proof}
This map is the composition of the bijections in~\cite[Def.~7.10]{BGNST1} and~\cite[Theorem 8.3]{NST_a}.
\end{proof}

As $\Comp_{r,n}$ and $\Part_{r,n}$
are both in bijection with $\binom{[n]}{r}$, we obtain a  direct bijection which appears to be new; see Figure~\ref{fig:qgrassmanian_objects_bij}.  
Define the \emph{outer hook length} $\OHL(\lambda)=\lambda_1+\ell(\lambda)-1$ (and $\OHL(\emptyset)=0$). 
\begin{figure}[!ht]
    \includegraphics[scale=0.8]{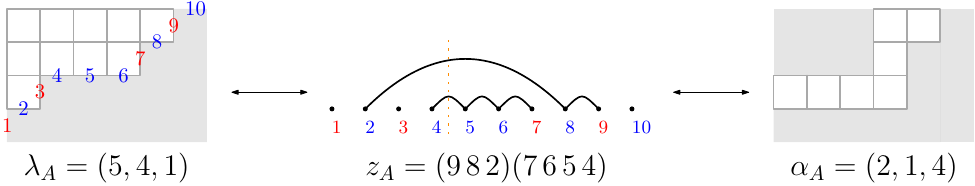}
    \caption{Bijections $\Part_{4, 10}\leftrightarrow \qgrass{4, 10} \leftrightarrow \Comp_{4, 10}$ for $A=\{1,3,7,9\}\subset [10]$.}
    \label{fig:qgrassmanian_objects_bij}
\end{figure}

\begin{thm}
\label{thm:parttocomp}
The composite bijection $\Comp_{r,n} \to \Part_{r,n}$, $\alpha \mapsto \lambda_{\alpha}$ is described as follows.
\leavevmode
\begin{enumerate}
\item Each $\alpha \in \Comp_{r,n}$ can be written as $\alpha = d_11^{e_1-1}(d_2+1)1^{e_2-1}\cdots (d_k+1)1^{e_k-1}$ for numbers $d_1,\ldots,d_k,e_1,\ldots,e_k\ge 1$.  (If $\alpha$ starts with $1^a$ then $d_1=1$ and $e_1=a$.) Using Frobenius notation for partitions, the bijection takes
\[
\alpha = d_11^{e_1-1}\cdots (d_k+1)1^{e_k-1}
\mapsto \lambda_{\alpha} \coloneq
\begin{pmatrix}(\sum_{i=1}^{k} d_{i})-1&(\sum_{i=2}^{k} d_{i})-1&\cdots & d_k-1\\ (\sum_{i=1}^{k} e_{i})-1 & (\sum_{i=2}^{k} e_{i})-1 &\cdots & e_k-1\end{pmatrix}.
\]
    
\item The inverse map $\Part_{r,n}\to \Comp_{r,n}$ takes a partition $\lambda=\begin{pmatrix}A_1&\cdots & A_k\\
L_1&\cdots & L_k\end{pmatrix}$ to $$(A_1-A_2)1^{L_1-L_2-1}(A_2-A_3+1)\cdots 1^{L_{k-2}-L_{k-1}-1}(A_{k-1}-A_k+1)1^{L_{k-1}-L_k-1}(A_k+2)1^{L_k}$$
if $k\ge 2$ and $(A_1+1)1^{L_1}$ if $k=1$.
\end{enumerate}
Under this bijection, if $\alpha$ corresponds to $\lambda$ then $|\alpha|=\OHL(\lambda)$, the outer hook length of $\lambda$. 
\end{thm}
\begin{proof}
We omit the proof of the bijections as it is a matter of carefully composing previous bijections. The union of blocks of $z_A$ for $A=(a_k,\ldots,a_1|_r b_1,\ldots,b_k)$ is the closed interval $[a_k,b_k]$, and the outer hook length of $\lambda_A$ is given by $b_k-a_k$  by \eqref{eqn:FrobSymbol}, so we conclude $|\alpha_A|=\operatorname{OHL}(\lambda)$ by \eqref{eq:qgrass_to_comp}.
\end{proof}

\subsection{Quasisymmetric inversions and noncrossing inversions}

We now relate the quasisymmetric inversions of $A\in \binom{[n]}{r}$ with a distinguished subset of ordinary inversions defined for $\NC_n$, following \cite[Section 8.3]{BGNST1}.
\begin{defn}\label{def:invnc}
    For $w\in \NC_n$, define the set of \emph{noncrossing inversions}  by
    \begin{align*}\invnc{w}\coloneqq \{(i,j)\in \inv{w}\suchthat (i,j)w\in \NC_n\}.
    \end{align*}
    Explicitly we have $(i,j)\in \invnc{w}$ if and only if $j$ is the largest element in its cycle and either
    \begin{enumerate}
        \item $i$ belongs to the same cycle as $j$, or
        \item $i<j<w^{-1}(i)$ and $i$ is maximal with respect to this property.  
    \end{enumerate}
\end{defn}
\begin{rem}
    Note that \cite[Section 8.3]{BGNST1} defines noncrossing inversions as right inversions, which in the present notation would be the set of $(w^{-1}(i),w^{-1}(j))=w^{-1}(i,j)w$ for $(i,j)\in \invnc{w}$. 
    The content of the second half of Definition~\ref{def:invnc} is then \cite[Proposition 8.16]{BGNST1}.
\end{rem}

\begin{lem}\label{le:invnc_za}
Let $A=(a_k,\dots,a_1 \mid_r b_1,\dots,b_k)$ and set $a_0\coloneqq r+1$, $b_0\coloneqq r$.
    The noncrossing inversions of $z_A$ are:
    \begin{enumerate}[label=(\arabic*)]
        \item \label{it1}
        \emph{(Split)} for $1\leq p\leq k$, $(a_p,b_p)$ \quad and \quad
        $(q,b_p)$ where $b_{p-1}<q<b_p$;
        \item\label{it2} 
        \emph{(Merge)}
         $(a_{p+1},b_p)$  for $1\leq p\leq k-1$,\quad and \quad
        $(a_{p+1},q)$ where  $a_{p+1}<q<a_{p}$ for $0\leq p\leq k-1$.
    \end{enumerate}
    In particular $|\invnc{z_A}|=\OHL(\lambda_A)$.
\end{lem}
\begin{proof}
    The combinatorial criterion in Definition~\ref{def:invnc} implies the description of the noncrossing inversions.
    As for the cardinality of $\invnc{z_A}$, observe that we have $b_k-r$ noncrossing inversions in item~\ref{it1} and $r-a_k$ noncrossing inversions in item~\ref{it2}. 
    Their sum $b_k-a_k$ equals $\OHL(\lambda_A)$.
\end{proof}

\begin{rem}
\label{rem:splitmerge}
    The labels refer to the effect of left-multiplying $z_A$ by the corresponding transposition: the inversions in~\ref{it1} split a cycle  into two, while those in~\ref{it2} merge two cycles  into one.   
\end{rem}

The description in Lemma~\ref{le:invnc_za} implies that $\invnc{z_A}\subseteq \inv{w_A}=\inv{A}$. 
Therefore the set $\invnc{z_A}$ identifies a distinguished collection of boxes of $\lambda_A$ under the usual correspondence between boxes and $\inv{w_A}=\inv{A}$. 
Let $k$ be the number of boxes along the main diagonal of $\lambda_A$.
Then Lemma~\ref{le:invnc_za} translates to the following:  noncrossing inversions in~\ref{it1} (resp.~\ref{it2})  correspond to the boxes along the main diagonal (resp. subdiagonal) of $\lambda_A$ as well as boxes in the first $k$ rows (resp. columns) that do not have a box below them (resp. to their right) in $\lambda_A$.

Figure~\ref{fig:ncinv_in_partition} lists the noncrossing inversions for the $z_A$ in Figure~\ref{fig:qgrassmanian_objects_bij} and identifies them with a subset of boxes in $\lambda_A$.
The dots (resp. crosses) in Figure~\ref{fig:ncinv_in_partition} represent inversions in ~\ref{it1} (resp.~\ref{it2}).

\begin{figure}[!ht]
    \centering
    \includegraphics[scale=0.8]{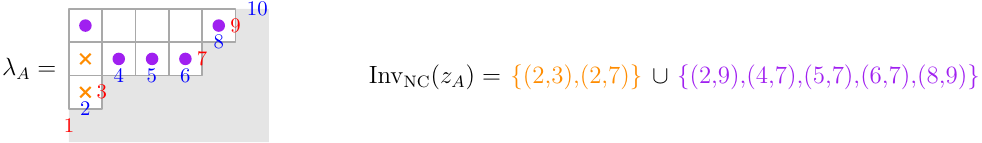}
    \caption{Noncrossing inversions of $z_A$ and the corresponding boxes in $\lambda_A$.}
    \label{fig:ncinv_in_partition}
\end{figure}

The next proposition cements the role of $r$-Quasigrassmannian permutations as the \textit{correct} noncrossing analogue of Grassmannian permutations.

\begin{thm}
   \label{thm:EdgesinQJ}
For $A\in \binom{[n]}{r}$, we have $|\invq{A}|=\operatorname{OHL}(\lambda_A)$ and
\[
    \invq{A}=\invnc{z_A}.
\]
\end{thm}
We shall prove this theorem after a preparatory lemma.
Say that a block $B$ of a noncrossing set partition \emph{extends over $r$} if $\min B\le r < \max B$.

\begin{lem}
\label{lem:DescribeNCToParts}
For $w\in \NC_n$ with associated noncrossing set partition $\mathcal{P}$, order the blocks extending over $r$ as $B_1,\ldots,B_k$ so that $a_i\coloneqq\max(B_i\cap[r])$ satisfies $a_1>\cdots>a_k$, and set $b_i\coloneqq\max B_i$. Then $b_1<\cdots<b_k$ and
$$w\cdot[r]=(a_k,\ldots,a_1\mid_r b_1,\ldots,b_k).$$
\end{lem}
\begin{proof}
First note if $j>i$ and $b_j<b_i$, then $a_j<a_i<b_j<b_i$ is a crossing between $B_i$ and $B_j$, a contradiction. Therefore  $b_1<\cdots<b_k$.

The backwards cycle of a block $B$ is supported on $B$, so blocks not extending over $r$ fix $[r]$ setwise. 
For a block $B_i$ extending over $r$ with $B\cap[r]=\{c_1<\cdots<c_p=a_i\}$ and $B\setminus[r]=\{d_1<\cdots<d_q=b_i\}$, the backwards cycle $(d_q,\ldots,d_1,c_p,\ldots,c_1)$ sends $c_j\mapsto c_{j-1}$ for $j\geq 2$ and $c_1\mapsto d_q$, so the image of $B_i\cap [r]$ is $((B_i\cap [r])\setminus a_i)\cup b_i$. 

Since distinct blocks act on disjoint sets, these replacements are independent, giving $w\cdot[r]=([r]\setminus\{a_1,\ldots,a_k\})\cup\{b_1,\ldots,b_k\}=(a_k,\ldots,a_1\mid_r b_1,\ldots,b_k)$ as desired.
\end{proof}

Call a block $B$ in a noncrossing set partition \emph{nested} if there exists another block $B'$ containing $i,j$ so that for all $k\in B$ we have $i<k<j$. If this happens we say $B$ is \emph{nested below} $B'$.

\begin{proof}[Proof of Theorem~\ref{thm:EdgesinQJ}]
The equality $|\invq{A}|=\operatorname{OHL}(\lambda_A)$ follows from Lemma~\ref{le:invnc_za} once we show $\invq{A}=\invnc{z_A}$.
As the noncrossing inversions described in Lemma~\ref{le:invnc_za} are a subset of the inversions described in Lemma~\ref{lem:invA}, we know $\invnc{z_A}\subset \inv{A}$. The inclusion $\invnc{z_A}\subseteq \invq{A}$ then follows as $z_A\in \NC_n^A$ and $$\invq{A}=\bigcup_{u\in \NC_n^A}(\invnc{u}\cap \inv{A}).$$ 

We prove the reverse inclusion. Let $(i,j)\in \invq{A}$, so there exists $w\in \NC_n^A$ with $(i,j)\in \invnc{w}\cap \inv{A}$. Fix such a $w$.
In what follows we let $B_s$ be the block of $w$ extending over $r$ with \[
(\max (B_s\cap [r]),\max B_s)=(a_s,b_s)\] from Lemma~\ref{lem:DescribeNCToParts}. 
By Lemma~\ref{le:invnc_za}, the element $j\in A$ is the maximal element of its cycle $B$.
There are two cases in this lemma and we address both.

Suppose first that $i\not\in A$ has $i\in B$. 
Then $B$ extends over $r$ as it is neither contained in $A$, nor is disjoint from $A$. 
This allows us to deduce that $B=B_p$ for some $1\le p \le k$, and that $j=b_p$. Furthermore, as $i\not\in A$ we know that $i\ge a_p$, so either $i=a_p$, in which case $(i,j)\in \invnc{z_A}$ by the first part of Lemma~\ref{le:invnc_za}\ref{it1}, or $r<i<b_p$ (as $a_p=\max B_p\cap [r]$). But in this latter case because $B_{p-1}$ is nested inside $B_p$ we know that $B_p\cap [a_{p-1},b_{p-1}]=\emptyset$ so in fact $b_{p-1}<i<b_p$, establishing $(i,j)\in \invnc{z_A}$ by the second part of Lemma~\ref{le:invnc_za}\ref{it1}.

Suppose now that $i\not\in A$ lies in a different cycle $B'$ to $B$, and is the maximal element such that $i<j<w^{-1}(i)$. 
For $j\in A$ we must have $\min B\le r$. The condition $i<j<w^{-1}(i)$ forces $B$ to be nested below $B'$, so we must have $i\le \min B$ and hence $i\le r$. 
As $i\not\in A$, we can only have $i\le r$ if $B'=B_{p+1}$ extends over $r$ for some $0\le p\le k-1$, and $i=a_{p+1}$. 
On the one hand if $B$ extends over $r$ then $B=B_{p'}$ for some $p'\le p$. 
But we cannot have $p'<p$ as then $a_p$ is a larger element than $i$ with $a_p<i<w^{-1}(a_p)$. 
We deduce $p'=p$ and so $s=b_p$, implying $(i,j)\in \invnc{z_A}$ by the first part of Lemma~\ref{le:invnc_za}\ref{it2}. 
If on the other hand $B$ does not extend over $r$ then we claim that $a_{p+1}<j<a_p$. 
Indeed, if not then $a_p<j\le r$ and we arrive at a contradiction that $i<a_p$ has $a_p<j<w^{-1}(a_p)$. 
We conclude from $a_{p+1}<j<a_p$ that $(i,j)\in \invnc{z_A}$ by the second part of Lemma~\ref{le:invnc_za}\ref{it2}.
\end{proof}

\begin{thm}
\label{thm:ZaInv}
For $A,B\in \binom{[n]}{r}$, we have $z_A\ge z_{B}$ in the Bruhat order if and only if $A\ge B$.
\end{thm}
\begin{proof}
For the order statement, recall that $x\ge y$ in Bruhat order if and only 
if $x\cdot[i]$ is elementwise at least as large as $y\cdot[i]$ for all 
$1\le i\le n$. Since $z_A\cdot[r]=A$ and $z_{A'}\cdot[r]=A'$, the condition 
at $i=r$ gives $z_A\ge z_{A'}$ whenever $A\ge A'$.

For the converse, it suffices to treat the case $B\lessdot A$, so that 
$AB\in QJ_{r,n}$ by Lemma~\ref{le:invnc_za} (see also the discussion after Remark~\ref{rem:splitmerge} where we can see clearly that each removable corner belongs to $\invq{A}$). By   Theorem~\ref{thm:EdgesinQJ} we conclude there exists $(i,j)\in \invnc{z_A}$ with 
$B=(A \setminus j) \cup i$. We then have
$(i,j)z_A\in\NC_n^{B}$,  so by minimality of 
$z_{B}\in \NC_n^{B}$ we conclude that $z_A>(i\,j)z_A\ge z_{B}$.
\end{proof}

\section{Pl\"{u}cker vanishing and $QJ_{r,n}$}
\label{sec:PlVan}

In this section we define two subsets
$$\pluckervanishing{G}\subseteq \bigsqcup_{A\in \binom{[n]}{r}} \mathring{X}^A_{\NC}\subset \Gr(r; n)$$ of the Grassmannian, and when $G=QJ_{r,n}$, we will show in the next section are both equal to $\QGr(r; n)$.

\subsection{Recollections on the Grassmannian}
\label{sec:RecollGrass}

The Grassmannian $\Gr(r; n)$ is the space of all $r$-planes in $\CC^{n}$, which we identify with its image under the \emph{Pl\"{u}cker embedding}
\[
\begin{array}{rcl}
\Gr(r; n) & \hookrightarrow & \PP(\bigwedge^{r}\CC^{n}) \\
V & \mapsto & \langle v_{1} \wedge v_{2} \wedge \cdots \wedge v_{r}\rangle
\end{array}
\qquad
\text{for $\{v_{1}, v_{2} \ldots, v_{r}\}$ a basis of $V$}.
\]
This map is equivariant with respect to the standard actions of $T = (\mathbb{C}^{\ast})^{n}$ on $\Gr(r;n)$ and $\mathbb{P}(\Lambda^r \mathbb{C}^n)$. For $A=\{i_1<\cdots<i_r\}$, let 
\[
e_{A} = e_{i_{1}} \wedge e_{i_{2}}\wedge \cdots \wedge e_{i_r} \in \textstyle \bigwedge^{r}\mathbb{C}^{n}.
\]  
The $T$-fixed points of $\Gr(r; n)$ are represented by the coordinate $r$-planes $\langle e_{A} \rangle$, and the $T$-invariant curves are $\langle e_{A}, e_{A'} \rangle$ for $AA'\in E(J_{r,n})$. We denote by $\{\Delta_A\suchthat A\in \binom{[n]}{r}\}$ the coordinate functions of the $e_A$, which are defined up to simultaneous scaling.

Each point $\langle v_{1} \wedge v_{2} \wedge \cdots \wedge v_{r}\rangle  \in \Gr(r; n)$ can be represented by the full-rank $n \times r$ matrix with columns $v_{1}, v_{2}, \ldots, v_{r}\in \mathbb{C}^n$, which is unique up to right multiplication by $\GL_{r}(\CC)$. Denoting $\operatorname{Mat}_{n\times r}^{\circ}\subset \operatorname{Mat}_{n\times r}$ for the subset of full column rank $n\times r$ matrices, we may thus identify $\Gr(r;n)=\operatorname{Mat}_{n\times r}^{\circ}/\GL_r(\mathbb{C})$.  
We will \textbf{transpose all $n \times r$ matrix representatives} for the sake of space, so the entry $(i,j)$ in the transposed matrix represents the entry in the $i$th column and $j$th row.

The Grassmannian $\Gr(r; n)$ has a stratification into Schubert cells, which we write as either $\mathring{X}^A=\mathring{X}^{\lambda}$ depending on whether we index with $A\in \binom{[n]}{r}$ or $\lambda\in \Part_{r,n}$.  
For each $A = \{a_1<\cdots<a_r\} \in \binom{[n]}{r}$, the Grassmannian Schubert cell $\mathring{X}^A$ is the subset represented by matrices with $1$'s in the positions $(a_i,i)$, $\ast$ (arbitrary entries) in the positions $(i,j)$ for $(i,a_j) \in \inv{A}$, and $0$'s elsewhere.  
The $\ast$ in $\mathring{X}^{\lambda_A} \coloneq \mathring{X}^A$ bijectively correspond to the boxes of the Young diagram for $\lambda_A$. 

\begin{eg}
\label{eg:1379NC}
For $n = 10$ and $r = 4$, let $A=\{1,3,7,9\}$ and $\lambda=(5,4,1)$. Then 
\[
\mathring{X}^{\{1, 3, 7, 9\}}
=
\begin{matrix}
1\\
3\\
7\\
9 
\end{matrix}
\begin{bmatrix}
1 & 0 & 0 & 0 & 0 & 0 & 0 & 0 & 0 & 0 \\
0 & \ast & 1 & 0 & 0 & 0 & 0 & 0 & 0 & 0 \\
0 & {\ast} & 0 & {\ast} & {\ast} & {\ast} & 1 & 0 & 0 & 0 \\
0 & {\ast} & 0 & {\ast} & {\ast} & {\ast} & 0 & {\ast} & 1 & 0
\end{bmatrix}^{\top}
=
\mathring{X}^{\tikz[scale = 0.25]{\foreach \y/\c in {1/5, 2/4, 3/1}{\foreach \x in {1,...,\c}{\draw (\x, -\y) rectangle (\x+1, 1-\y);}}}}
\cong \mathbb{A}^{10}.
\]
\end{eg}

In terms of a matrix representative $M$ of a point in $\Gr(r; n)$, the Pl\"{u}cker coordinates $\Delta_{A}$ can be simultaneously computed as $\Delta_A=\det (M_{i, j})_{i \in A,\, j \in [r]}$.  This gives the following result.

\begin{fact}
\label{fact:boringdeterminant}
At any point of $\mathring{X}^A$ we have $\Delta_{A'}=0$ if $A'\not\le A$, $\Delta_{A}\neq 0$, and if $(i,a_j)\in \inv{A}$ and we set $A'=(A\setminus a_j)\cup i$ then $\Delta_{A'}\big/\Delta_{A}=\pm M_{i,j}$ for $M$ the canonical matrix representative of our point.
\end{fact}

Recall that an \emph{affine paving} of a variety $X$ is a sequence  $\emptyset=X_0\subset X_1\subset \cdots \subset X_k=X$ of closed subvarieties such that each $X_i\setminus X_{i-1}$ is isomorphic to an affine space. 
We have $\mathbb{A}^{|\inv{A}|}\cong \mathring{X}^A=\mathring{X}^{\lambda_A}\cong \mathbb{A}^{|\lambda_A|}$, and the Schubert cells listed in any linear extension of the Gale order induce an affine paving 
\[
\Gr(r; n) = \bigsqcup_{A \in \binom{[n]}{r}} \mathring{X}^A = \bigsqcup_{\lambda \in \Part_{r, n}} \mathring{X}^{\lambda}.
\]

\subsection{Pl\"{u}cker vanishing and Quasisymmetric Schubert cells}
\label{sec:PVXNC}
For any subgraph $G \subseteq J_{r, n}$, we define an associated \emph{Pl\"ucker vanishing variety} by
\[
\pluckervanishing{G}\coloneqq \left(\bigcap_{A\in V(J_{r,n})\setminus V(G)} \{\Delta_A=0\} \right)\cap\left(\bigcap_{AB\in E(J_{r,n})\setminus E(G)}\{\Delta_A\Delta_{B}=0\}\right)\subset \Gr(r;n).
\]

Write $\operatorname{Inv}^G(A)$ for the set of inversions $(i,j)$ of $A$ such that for $B=(A\setminus j)\cup i$ we have $AB\in E(G)$.
From its definition there is a natural union of affine subspaces of the Schubert cells which always contains $\pluckervanishing{G}$. For $A = \{a_{1} < \cdots < a_{r}\}\in V(G)$, let 
\[
\mathring{X}^A_G \coloneqq \{M \in \mathring{X}^A \suchthat \text{$M_{i,j} = 0$ for $(i,a_j) \in \inv{A}\setminus \operatorname{Inv}^G(A)$}\}\cong \mathbb{A}^{|\operatorname{Inv}^G(A)|}.
\]
See Example~\ref{eg:A356} for the application of this construction to $G = QJ_{r, n}$.

\begin{prop}
\label{prop:QGrcontainedPV}
    We have $\mathring{X}^A_G=\mathring{X}^A\cap \bigcap_{(i,j)\in \operatorname{Inv}(A)\setminus \operatorname{Inv}^G(A)}\{\Delta_{(A\setminus j)\cup i}=0\}$ and
$$\pluckervanishing{G}\subseteq \bigsqcup_{A\in V(G)}\mathring{X}^A_G\subset \Gr(r;n)$$
\end{prop}
\begin{proof}
The equality follows directly from Fact~\ref{fact:boringdeterminant}. We now show the containment of $\pluckervanishing{G}$.
    Since $\Delta_A\ne 0$ on $\mathring{X}^A$, we have $\pluckervanishing{G}\subset \bigcup_{A\in V(G)}\mathring{X}^A$, so it suffices to show for $A\in V(G)$ that $\pluckervanishing{G}\cap \mathring{X}^A\subset \mathring{X}^A_G$. By  Fact~\ref{fact:boringdeterminant} we have $\Delta_A\ne 0$ on $\mathring{X}^A$, so for $(i,j)\in \operatorname{Inv}(A)\setminus \operatorname{Inv}^G(A)$ we have $\Delta_{(A\setminus j)\cup i}= 0$ is equivalent to $\Delta_A\Delta_{(A\setminus j)\cup i}=0$ on $\mathring{X}^A$, and this is a defining relation for $\pluckervanishing{G}$. 
\end{proof}

Now consider $G=QJ_{r,n}$. Then $\operatorname{Inv}^G(A)=\invq{A}$ has size $\OHL(\lambda)$ (Theorem~\ref{thm:EdgesinQJ}) and
\[
\pluckervanishing{QJ_{r,n}}=\bigcap_{AA'\in E(J_{r,n})\setminus E(QJ_{r,n})}\{\Delta_A\Delta_{A'}=0\}.
\]
\begin{defn}
\label{de:quasisymmetric_schubert_cell}
Define the \emph{quasisymmetric Schubert cell} $\mathring{X}^A_{\NC}\cong \mathbb{A}^{|\OHL(\lambda_A)|}$ to be the subspace of $\mathring{X}^A$ obtained by setting to zero all matrix entries $M_{j,i}$ for $(a_i,j)\in \inv{A}\setminus\invq{A}$.  
\end{defn}
 Equivalently $\mathring{X}^A_{\NC}=\mathring{X}^A_G$ for $G=QJ_{r,n}$.  
Under the correspondence between boxes of $\lambda$ and $\ast$'s in $\mathring{X}^{\lambda}$, these entries are characterized by the results of Section~\ref{sec:QJn}.

\begin{eg}
\label{eg:A356}
Continuing Example~\ref{eg:1379NC}, we have $\inv{A}\setminus\invq{A}=\{(4,9),(5,9),(6,9)\}$ corresponding to the unfilled cells in Figure~\ref{fig:ncinv_in_partition},
and additionally







$$
\mathring{X}^A_{\NC}\ = \ 
\begin{matrix} 1\\3\\ 7\\ 9 \end{matrix}
\begin{bmatrix}1&0&0&0&0&0&0&0&0&0\\
0&\ast&1&0&0&0&0&0&0&0\\
0&\ast&0&\ast&\ast&\ast&1&0&0&0\\
0&\ast&0&\boxed{0}&\boxed{0}&\boxed{0}&0&\ast&1&0\end{bmatrix}^{\top}\cong \mathbb{A}^{7}.
$$





\end{eg}
\begin{cor}
    We have
    $$\pluckervanishing{QJ_{r,n}}\subseteq \bigsqcup_{A\in \binom{[n]}{r}}\mathring{X}^A_{\NC}\subset \Gr(r;n).$$
\end{cor}
What is special about $QJ_{r,n}$ is that this first containment is an equality, and we will show in the next section that these are equal to the quasisymmetric Grassmannian $\QGr(r;n)$.

\section{The Quasisymmetric Grassmannian}
\label{sec:QGr}

The complete flag variety $\fl{n}$ parametrizes flags of subspaces $\{0\subsetneq V_1\subsetneq \cdots \subsetneq V_{n-1}\subsetneq \mathbb{C}^n\}$.  
The \emph{quasisymmetric flag variety} of~\cite{BGNST1} is a complex of projective toric varieties $\hhmp_n\subset \fl{n}$.  
The natural projection map $\pi : \fl{n} \to \Gr(r; n)$ sends the flag $V_{\bullet}$ to the subspace $V_{r}$.
Recalling Definition~\ref{introdef_qgr}, the \emph{quasisymmetric Grassmannian} is defined to be the projection
\[
\QGr(r;n)\coloneqq \pi(\hhmp_n).
\]
Since $\pi$ is proper and $T$-equivariant, $\QGr(r;n)\subset \Gr(r;n)$ is also a complex of projective toric varieties.  
Recall the sets $\pluckervanishing{QJ_{r,n}}$ and $\mathring{X}^\lambda_{\NC}$ from Section~\ref{sec:PVXNC}.

\begin{thm}
\label{thm:everythingequal}
We have the equalities
\[
\QGr(r;n)= \pluckervanishing{QJ_{r,n}}=\bigsqcup \mathring{X}^\lambda_{\NC}=\bigcup X^\lambda_{\NC}.
\]
Furthermore the quasisymmetric Schubert cells $\mathring{X}^\lambda_{\NC}$ give an affine paving of $\QGr(r;n)$ under any linear extension of the Gale order $\le$.
\end{thm}

The proof of Theorem~\ref{thm:everythingequal} is given at the end of the section; we first recall an affine paving of $\hhmp_{n}$ in Section~\ref{sec:qfl_recall} and relate this paving to the cells $\mathring{X}^\lambda_{\NC}$ in Section~\ref{sec:projection_of_cells}.   

\subsection{Noncrossing Schubert cells in $\hhmp_n$}
\label{sec:qfl_recall}

We identify $\fl{n}$ with $\GL_n/B$, where $B\subset \GL_n$ is the subgroup of upper triangular matrices.  
The flag associated to a coset $MB \in \GL_{n}/B$ is obtained by taking $V_i$ to be the span of the first $i$ columns of $M$.  
The matrix representative of the projection $\pi(MB) \in \Gr(r; n)$ is the principle $n \times r$ submatrix $(M_{i, j})_{\substack{1 \le i \le n \\ 1 \le j \le r}}$, or any right $\GL_{r}$-multiple thereof. 

For $w \in S_{n}$, the \emph{Schubert cell} $\mathring{X}^w$ is $BwB/B$.  The Bruhat decomposition states
\[
\GL_n/B=\bigsqcup_{w\in S_n}\mathring{X}^w\quad\text{ with }\quad \mathring{X}^w\cong \mathbb{A}^{\ell(w)},
\]
and this gives an affine paving of $\fl{n}$ under any linear extension of the Bruhat order.  
Furthermore every Schubert variety $X^{w} = \overline{\mathring{X}^{w}}$ is a union of Schubert cells, though this is not required by the definition of an affine paving.  

Each $\mathring{X}^w$ has a canonical affine chart that can be represented as a $\{0, 1, \ast\}$ matrix in the manner of Section~\ref{sec:RecollGrass}.  
Beginning with the permutation matrix for $w$, which has $1$'s in positions $(w(j), j)$ and $0$'s elsewhere, fill each position which is not  below or right of a $1$ with  $\ast$, representing a free coordinate.  
This amounts to placing an $\ast$ in position $(i,w^{-1}(j))$ for each inversion $(i, j) \in \inv{w}$.

Following~\cite[Definition 9.6]{BGNST2}, define the \emph{noncrossing Schubert cell}\footnote{Noncrossing Schubert cells are called noncrossing Bruhat cells in~\cite{BGNST2}.} $\mathring{X}^w_{\NC}$ for $w\in \NC_n$ to be the subspace of $\mathring{X}^{w}$ obtained by zeroing out the $\ast$ in positions $(i, w^{-1}(j))$ for $(i, j) \notin \invnc{w}$.  
Also define the \emph{noncrossing Schubert cycle} $X^w_{\NC} = \overline{\mathring{X}^w_{\NC}}$. 

\begin{eg}
Let $w = 53214$, so that $\invnc{w} = \{(1\,3), (1\,5), (2\,3), (4\,5)\}$ and $\inv{w} \setminus \invnc{w} = \{(1\,2), (2\,5), (3\,5)\}$.  Then we have
\[
\mathring{X}^{53214}=\begin{bmatrix}
\ast&\ast&\ast&1&0\\
\ast&\ast&1&0&0\\
\ast&1&0&0&0\\
\ast&0&0&0&1\\
1&0&0&0&0
\end{bmatrix}\cong\mathbb{A}^7\quad\text{ and }\quad\mathring{X}^{53214}_{\NC}=\begin{bmatrix}
\ast&\ast&\boxed{0}&1&0\\
\boxed{0}&\ast&1&0&0\\
\boxed{0}&1&0&0&0\\
\ast&0&0&0&1\\
1&0&0&0&0
\end{bmatrix}\cong \mathbb{A}^4.
\]
\end{eg}
In \cite{BGNST2} these cells were shown to induce an affine paving of $\hhmp_n$ with respect to any linear extension of the Bruhat order restricted to $\NC_n$:
$$
\hhmp_n=\bigsqcup_{w\in \NC_n}\mathring{X}^w_{\NC}\quad\text{ with }\quad\mathring{X}^w_{\NC}\cong \mathbb{A}^{|\invnc{w}|}.
$$
Unlike the paving of $\fl{n}$ by Schubert cells, some of the cell closures $X^w_{\NC}$ are strictly contained in (and are not equal to) a union of noncrossing Schubert cells.

\subsection{Projection of noncrossing Schubert cells in $\fl{n}$}
\label{sec:projection_of_cells}

For $w \in S_{n}$, if $w \cdot [r] = A = \{a_{1} < \cdots < a_{r}\}$ then the projection $\pi$ maps $\mathring{X}^{w}$ onto $\mathring{X}^{A}$.  
More specifically, $\inv{A} \subseteq \inv{w}$ and $\pi|_{\mathring{X}^{w}}$ is the coordinate projection sending the $(i, w^{-1}(j))$-coordinate subspace of $\mathring{X}^{w}$ to zero unless $w^{-1}(j)\le r < w^{-1}(i)$, in which case $(i, j) = (i, a_{k}) \in \inv{A}$ and the $(i, w^{-1}(j))$-coordinate subspace maps to the $(i, k)$-coordinate subspace of $\mathring{X}^{A}$.

\begin{thm}
\label{thm:posfibers}
If $u \in \NC_{n}^{A}$, then $\pi(\mathring{X}^{u}_{\NC}) \subseteq \mathring{X}^{A}$.  Moreover: 
\begin{enumerate}
\item if $u = z_{A}$, then $\pi|_{\mathring{X}^{u}_{\NC}}$ is an isomorphism onto $\mathring{X}^{A}_{\NC}$; and

\item if $u \neq z_{A}$ then $\pi|_{\mathring{X}^{u}_{\NC}}$ has positive-dimensional fibers. 
\end{enumerate}
\end{thm}
\begin{proof}
By definition $u \in \NC_{n}^{A}$ has $u \cdot [r] = A$, so the remarks above imply that $\pi(\mathring{X}^{u}_{\NC}) \subseteq \mathring{X}^{A}$. 

If $u = z_{A}$, Theorem~\ref{thm:EdgesinQJ} tells us that $\invnc{u} = \invq{A}$, so $\mathring{X}^{u}_{\NC}$ contains none of the linear subspaces which generate $\ker(\pi|_{\mathring{X}^{u}})$ and maps onto $\mathring{X}^{A}_{\NC}$.

If $u \neq z_{A}$, we show that $X^{u}_{\NC}$ intersect the kernel $\ker(\pi|_{X^{u}})$ nontrivially.  
Indeed, by Lemma~\ref{le:welldefined} there is some $i \in \desnc{u} \setminus \{r\}$, and $(u(i+1), u(i)) \in \invnc{u}$ with $u(i+1), u(i) \le r$ or $r < u(i+1), u(i)$.  Then the corresponding coordinate subspace in $\mathring{X}^{u}_{\NC}$ is in the kernel of $\pi|_{X^{u}}$.  
\end{proof}
\begin{eg}
Let $\lambda=(3,3,2)$. Then $A=\{3,5,6\}$,  $z_A=653241$, and the noncrossing inversions are $\invnc{z_A}=\{(1,6),(1,5),(4,5),(2,5),(2,3)\}$. We then compute $$\{(i,z_A^{-1}(j))\suchthat (i,j)\in \invnc{z_A}\}=\{(1,1),(1,2),(4,2),(2,2),(2,3)\},$$ so the associated noncrossing Schubert cell is given by
$$\mathring{X}^{z_A}_{\NC}=\begin{bmatrix}
\ast&\ast&\boxed{0}&\boxed{0}&\boxed{0}&1\\
\boxed{0}&\ast&\ast&1&0&0\\
\boxed{0}&\boxed{0}&1&0&0&0\\
\boxed{0}&\ast&0&0&1&0\\
\boxed{0}&1&0&0&0&0\\
1&0&0&0&0&0\end{bmatrix}\cong \mathbb{A}^5.$$
The column span of the first $3$ columns then gives the isomorphic projection to $\mathring{X}^{\lambda}_{\NC}\subset\Gr(3;6)$. This agrees with the earlier construction: $\inv{A}\setminus\invq{A}=\{(2,6),(4,6),(1,3)\}$ and
$$
\begin{tikzpicture}[scale=0.5,baseline=-.3cm]

\node[above] at (.5,1.1) {1};
\node[above] at (1.5,1.1) {2};
\node[above] at (2.5,1.1) {4};

\node[left] at (-.1,.5) {6};
\foreach \x in {0,1,2} {
    \draw (\x,0) rectangle ++(1,1);
}

\node[left] at (-.1,-.5) {5};
\foreach \x in {0,1,2} {
    \draw (\x,-1) rectangle ++(1,1);
}

\node[left] at (-.1,-1.5) {3};
\draw (0,-2) rectangle ++(1,1);
\draw (1,-2) rectangle ++(1,1);

\node at (0.5,0.5) {$\bullet$};
\node at (0.5,-0.5) {$\bullet$};
\node at (1.5,-0.5) {$\bullet$};
\node at (1.5,-1.5) {$\bullet$};
\node at (2.5,-0.5) {$\bullet$};
\end{tikzpicture}
\qquad\to\qquad
\mathring{X}^A_{\NC}\ = \ 
\begin{matrix} 3\\ 5\\ 6 \end{matrix}
\begin{bmatrix}\boxed{0}&\ast&1&0&0&0\\
\ast&\ast&0&\ast&1&0\\
\ast&\boxed{0}&0&\boxed{0}&0&1\end{bmatrix}^{\top}\cong \mathbb{A}^{5}.
$$
\end{eg}

We now consider the (noncrossing) Schubert cells under the action of the subgroup $T \subseteq \GL_{n}$ of diagonal matrices.  
We will denote by $\chi_i$ the character of $T$ that  maps $\mathrm{Diag}(h_{1}, \ldots, h_{n}) \in T$ to $h_{i}$, and write $\mathbb{C}_{\chi}$ for the corresponding one-dimensional representation. 

The left action of $T$ on $\GL_{n}/B$ makes $\mathring{X}^w$ into a linear $T$-representation, and this action scales the $\ast$ in entry $(i,w^{-1}(j))$ by the character $\chi_{i}/\chi_{j}$.  We therefore have 
\[
\mathring{X}^w\cong \bigoplus_{(i,j)\in \inv{w}}\mathbb{C}_{\chi_i/\chi_j}
\qquad\text{and}\qquad
\mathring{X}^u_{\NC}\cong \bigoplus_{(i,j)\in \invnc{u}}\mathbb{C}_{\chi_i/\chi_j}.
\]
The $T$ action on Grassmannian Schubert cells is similar: for $A \in \binom{[n]}{r}$,  $T$ acts on $\mathring{X}^A$ 
by scaling the $\ast$ corresponding to $(i, j) \in \inv{A}$ by $\chi_{i}/\chi_{j}$, so that
\[
\mathring{X}^A\cong \bigoplus_{(i,j)\in \inv{A}}\mathbb{C}_{\chi_i/\chi_{j}}
\qquad\text{and}\qquad
\mathring{X}^A_{\NC}\cong \bigoplus_{(i,j)\in \invq{A}}\mathbb{C}_{\chi_i/\chi_j}.
\]
The projection map $\pi:\fl{n}\to \Gr(r;n)$ induces maps $\mathring{X}^w\to \mathring{X}^{w\cdot [r]}$ and $\mathring{X}^u_{\NC}\to \mathring{X}^A_{\NC}$ by sending $\mathbb{C}_{\chi_{i}/\chi_{j}}$ to either $\mathbb{C}_{\chi_{i}/\chi_{j}}$ or $0$ according to whether $(i,j)\in \inv{A}$.

\begin{proof}[Proof of Theorem~\ref{thm:everythingequal}]
By Proposition~\ref{prop:QGrcontainedPV} and Theorem~\ref{thm:posfibers} we know that 
\[
\pluckervanishing{QJ_{r,n}} 
\subseteq 
\bigsqcup \mathring{X}^\lambda_{\NC} 
\subseteq 
\bigcup X^\lambda_{\NC}
\subseteq 
\QGr(r;n),
\]
so what remains is to show that $\QGr(r;n)\subseteq \pluckervanishing{QJ_{r,n}}$.

We start by showing that any $1$-dimensional $T$-orbit closure in $\QGr(r;n)$ is of the form $\langle e_A,e_{B}\rangle$ with $AB\in E(QJ_{r,n})$. 
By definition $
\QGr(r;n)=\pi(\hhmp_n)=\bigcup_{u\in \NC_n} \pi(\mathring{X}^u_{\NC})$.    
Since no two characters of $\mathring{X}^A$ are linearly dependent, any $1$-dimensional $T$-orbit in $\pi(\mathring{X}^u_{\NC})$ is the $T$-orbit of an interior point $x$ on a coordinate line $\langle e_{A}, e_{B} \rangle$. 
Moreover $\pi|_{\mathring{X}^u_{\NC}}$ is a coordinate projection onto a coordinate subspace of $\mathring{X}^A_{\NC}$ for $A = u \cdot [r]$.
Using the coordinate projection property we can find a point $z\in \mathring{X}^u_{\NC}$ such that $\pi(z)=x$ and $z$ lies in the interior of a coordinate line connecting $u$ to some $(i\,j)u$ with $(i,j)\in \invnc{u}$. 
We therefore have 
\[
\overline{T\cdot x}=\langle e_A,e_{B}\rangle, \qquad A=u \cdot [r],\quad B=(i\,j)u\cdot [r],
\]
 and $(i, j) \in \invq{A}$ by the definition of $\mathring{X}^A_{\NC}$.

To show $\QGr(r;n)\subset \pluckervanishing{QJ_{r,n}}$, by the definition of $\pluckervanishing{QJ_{r,n}}$ it suffices to show that for any $x\in \QGr(r;n)$ and any edge $AB\in E(J_{r,n})\setminus E(QJ_{r,n})$, we have $(\Delta_A\Delta_{B})(x)=0$.
Suppose on the contrary that $(\Delta_A\Delta_{B})(x)\ne 0$. 
Writing $B=(A\setminus j)\cup i$, consider the cocharacter $\rho: \mathbb{C}^{\ast} \to T$ defined by
\[
\rho(t)=(t^{-b_1},\ldots,t^{-b_n}) 
\qquad \text{where}\qquad
b_k=\delta_{k\in A}+\delta_{k\in B}.
\]
Then $x'=\lim_{t\to 0}\rho(t)x\in \QGr(r;n)$ because $\QGr(r;n)=\pi(\hhmp_n)$ is closed and $T$-invariant (as $\pi$ is proper and $T$-equivariant). 
This $x'$ lies in the interior of $\langle e_A,e_{B}\rangle$, as in $\mathbb{P}(\Lambda^r\mathbb{C}^n)$ the cocharacter $\rho(t)$ scales each $e_{C}$ by $t^{-N_{C}}$ with 
\[
N_{C} =2|C\cap (A\cap B)|+|C \cap \{i,j\}|,
\]
a quantity that is maximized simultaneously at $C\in \{A,B\}$. 
But then $\langle e_A,e_{B}\rangle=\overline{T\cdot x'}\subset \QGr(r;n)$, contradicting the fact that every $1$-dimensional $T$-orbit closure in $\QGr(r;n)$ satisfies $AB\in E(QJ_{r,n})$.

For the affine paving statement, since  $\mathring{X}^\lambda_{\NC}\cong \mathbb{A}^{\OHL(\lambda)}$ it suffices to show that under the total ordering $\emptyset=\lambda_1,\lambda_2,\ldots$ we have $X_k\coloneqq \bigsqcup_{i\le k} \mathring{X}^{\lambda_i}_{\NC}$ is closed. This follows as $$X_k=\QGr(r;n)\cap \bigsqcup_{i=1}^k \mathring{X}^{\lambda_i}=\QGr(r;n)\cap\bigcup_{i=1}^k X^{\lambda_i},$$
where the last equality uses the fact that $X^{\lambda}=\bigsqcup_{\lambda'\subset \lambda} \mathring{X}^{\lambda'}$.
\end{proof}
See Figure~\ref{fig:paving_qgr24} which shows the affine paving 
\[
\QGr(2;4)=\underbrace{\mathring{X}^{12}}_{\mathbb{A}^0}\sqcup \underbrace{\mathring{X}^{13}}_{\mathbb{A}^1}\sqcup \underbrace{\mathring{X}^{14}}_{\mathbb{A}^2}\sqcup \underbrace{\mathring{X}^{23}}_{\mathbb{A}^2}\sqcup \underbrace{\mathring{X}^{24}}_{\mathbb{A}^3}\sqcup \underbrace{\begin{bmatrix}\ast&\ast&1&0\\\ast&\boxed{0}&0&1\end{bmatrix}^{\top}}_{\mathring{X}^{34}_{\NC}\cong \mathbb{A}^3}
\]
in terms of moment polytopes: the closures  of the top-dimensional cells $X^{24}_{\NC}=X^{24}$ and $X^{34}_{\NC}$ correspond respectively to the blue and yellow pyramids. 
\begin{figure}[!ht]
    \centering
    \includegraphics[width=\linewidth]{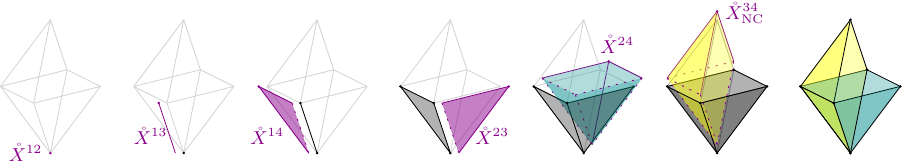}
    \caption{Paving $\QGr(2;4)$ with quasisymmetric Schubert cells. The $T$-orbits added by the affine cell $\mathring{X}^{ij}_{\NC}$ correspond to faces of the moment polytope for $\mathring{X}^{ij}_{\NC}$ that touch vertex $ij$.  The union of the new faces is homeomorphic to a half-open cube.}
    \label{fig:paving_qgr24}
\end{figure}

\section{Positroids, Richardsons, and Combinatorial Rigidity}
\label{sec:Positroid}
In this section we connect the quasisymmetric Schubert cycles to positroid varieties and Richardson varieties. Using these connections we show that the complex of moment polytopes determines the inclusion relations between the $T$-orbit closures in $\QGr(r;n)$.

\subsection{Positroids}
We now recall both geometric and combinatorial notions pertaining to total positivity as first described by Postnikov \cite{Po06} and further developed by Knutson--Lam--Speyer \cite{KLS13}; see also~\cite{Po18,Wi23} for a survey of the various mathematical fields this touches.
\begin{defn}
    A matrix $M\in Mat_{n\times r}^{\circ}$ with real entries is \emph{totally nonnegative} if all its maximal minors satisfy $\Delta_A(M)\ge 0$.
    Its \emph{positroid} is
\[
\mathcal{A}(M) \coloneqq \{\,A\in \tbinom{[n]}{r} \suchthat \Delta_A(M)>0\,\}.
\]
A subset $\mathcal{A}\subseteq \binom{[n]}{r}$ is a \emph{positroid} if $\mathcal{A}=\mathcal{A}(M)$ for some totally nonnegative $M$.
    The \emph{totally nonnegative Grassmannian} $\Gr(r;n)_{\ge 0}$ is the set of real linear spaces determined by the column spans of the totally nonnegative matrices. A \emph{positroid variety} is the closure in the complex Grassmannian $\Gr(r;n)$ of the points in $\Gr(r;n)_{\ge 0}$ with the same positroid.
\end{defn}

Positroids admit a combinatorial description in terms of Le diagrams \cite[Definition 6.1]{Po06}.
We isolate a subset of such diagrams that is relevant to us.
\begin{defn}
A subset $L \subseteq \lambda$ is called a \emph{noncrossing Le diagram} if it satisfies the following two conditions:
\begin{enumerate}
\item (\emph{Le condition}) If a square in $\lambda$ lies strictly below a square of $L$ in the same column and strictly to the right of a square of $L$ in the same row, then it also belongs to $L$.
\item (\emph{noncrossing condition}) For every box $b \in L$, either there is no box of $L$ strictly above $b$ in its column, or there is no box of $L$ strictly to the left of $b$ in its row.
\end{enumerate}
\end{defn}

\begin{rem}
    Every noncrossing Le diagram determines a unique \emph{noncrossing alternating forest} in the sense of~\cite{AN09}. 
    Suppose that $L\subset \lambda_A$ for $\lambda_A\in \Part_{r,n}$ and consider the subset of $\inv{A}$ determined by the boxes of $L$.  If we represent each inversion $(i, j)$ by an arc from $i$ to $j$, as shown in Figure~\ref{fig:ncle_to_ncaf}, the result is a noncrossing alternating forest.  
    If $L$ is maximal under inclusion, then the corresponding forest will be a \emph{noncrossing alternating tree} \cite{GGP97,Pos09} with $\OHL(\lambda)$-many arcs. 
\end{rem}

\begin{figure}[!ht]
    \centering
    \begin{tikzpicture}[baseline = -2cm]
    \begin{scope}[scale = 0.5, yshift = 2cm, xshift = 2.5cm]
    \path[fill = gray!30] (0, 0) rectangle (6, -4);
    \foreach \r\c in {1/5, 2/4, 3/1}{
        \foreach \i in {1,...,\c}{
            \draw[color = gray, fill=white] (\i-1, 1-\r) rectangle (\i, -\r);
        }
    }
    \foreach \x/\y/\l in {6/0/10, 5/1/8, 4/2/6, 3/2/5, 2/2/4, 1/3/2}{
        \draw (\x-0.5, -\y) node (\l) {};
        \draw (\l) node[color = blue] {$\scriptstyle \l$};
    }
    \foreach \x/\y/\l in {5/0/9, 4/1/7, 1/2/3, 0/3/1}{
        \draw (\x, -\y-0.5) node (\l) {};
        \draw (\l) node[color = red] {$\scriptstyle \l$};
    }
    \foreach \x/\y/\l in {1/0/in11, 1/1/in12, 1/2/in13, 2/1/in22, 3/1/in32, 4/1/in42, 5/0/in51}{
        \draw[fill] (\x-0.5, -\y-0.5) circle (3pt) node (\l) {};
    }
    \end{scope}
    \begin{scope}[scale = 0.5, yshift = -6cm]
    \foreach \x in {1, 3, 7, 9}{
        \draw[color=red] (\x, 0) node (\x) {$\scriptstyle \x$};
        }
    \foreach \x in {2, 4, 5, 6, 8, 10}{
        \draw[color=blue] (\x, 0) node (\x) {$\scriptstyle \x$};
        }
    \foreach \i/\j in {2/3, 2/7, 2/9, 4/7, 5/7, 6/7, 8/9}{
        \draw (\i) to[out = 90, in = 90] (\j);
        }
    \end{scope}
    \begin{scope}[scale = 1, xshift = 7cm, yshift = 0.5cm]
    \draw[dotted] (-0.5, -3.75) -- (-0.5, 0.5) -- (5.75, 0.5) -- (5.75, 0) -- (5, 0) -- (5, -1) -- (4, -1) -- (4, -2) -- (1, -2) -- (1, -3) -- (0, -3) -- (0, -3.75) -- (-0.5, -3.75);
    \foreach \x/\y/\l in {6/0/10, 5/1/8, 4/2/6, 3/2/5, 2/2/4, 1/3/2}{
        \draw[color = blue, fill = blue] (\x-0.5, -\y) circle (2pt) node (\l) {};
        \draw (\l) node[below, color = blue] {$\scriptstyle \l$};
    }
    \foreach \x/\y/\l in {5/0/9, 4/1/7, 1/2/3, 0/3/1}{
        \draw[color = red, fill = red] (\x, -\y-0.5) circle (2pt) node (\l) {};
        \draw (\l) node[right, color = red] {$\scriptstyle \l$};
    }
    \foreach \x/\y/\l in {1/0/in11, 1/1/in12, 1/2/in13, 2/1/in22, 3/1/in32, 4/1/in42, 5/0/in51}{
        \draw[fill] (\x-0.5, -\y-0.5) circle (2pt) node (\l) {};
    }
    \foreach \s/\t in {9/in51, in51/8, in51/in11, in11/in12, 7/in42, in42/6, in42/in32, in32/5, in32/in22, in22/4, in22/in12, in12/in13, 3/in13, in13/2}{
        \draw[thick, -{Latex[length=1mm,width=1.5mm]}] (\s) -- (\t);
    }
    \end{scope}
    \end{tikzpicture}
    \caption{Clockwise from the top left: a noncrossing Le diagram $L \subseteq \lambda = (5, 4, 1)$ marked by $\bullet$, the $\Gamma$-network for $L$, and the noncrossing alternating forest for $L$.}
    \label{fig:ncle_to_ncaf}
\end{figure}
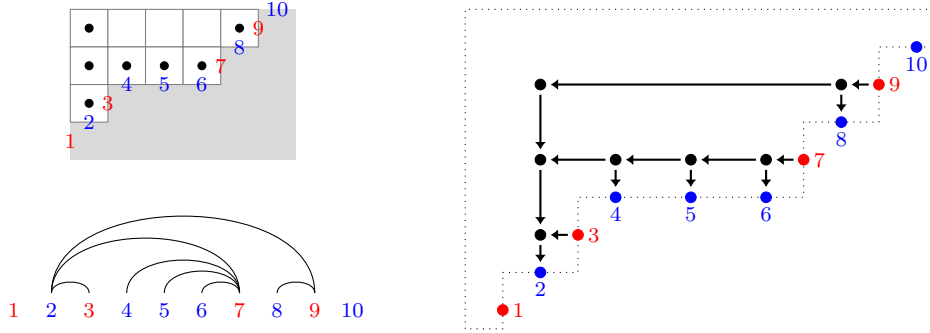

\begin{prop}
    For a noncrossing Le diagram $L$, let $\mathring{X}^\lambda_{L} \subseteq \mathring{X}^\lambda$ be the subspace with all coordinates outside $L$ equal to zero.  
    Then $X^{\lambda}_{L} = \overline{\mathring{X}^\lambda_{L}}$ is the positroid variety with Le diagram $L$.
\end{prop}
\begin{proof}
Each Le diagram defines an acyclic directed planar ``$\Gamma$-network'' on the vertices $[n] \cup L$; see Figure~\ref{fig:ncle_to_ncaf}. Recall by Fact~\ref{fact:boringdeterminant} that in the coordinitization of $\mathring{X}^\lambda$, the $\ast$ corresponding to the inversion $(i,j)\in \inv{A_\lambda}$ is given by $\pm \Delta_{(A_{\lambda} \setminus j) \cup i}/\Delta_{A_{\lambda}}$. The positroid variety associated to $L$ is the closure of a subset of $\mathring{X}^\lambda$  parametrized by $\vec{c} \in (\CC^*)^{L}$, whose coordinates are uniquely determined by
\[
\Delta_{(A_{\lambda} \setminus j) \cup i}/\Delta_{A_{\lambda}} = \sum_{P \::\: j \to i} \mathrm{wt}_{P}(\vec{c}),
\]
where the sum above is over directed walks $P$ from $j$ to $i$ in the network and $\mathrm{wt}_{P}(\vec{c})$ is the product of $c_{b}$ for each $b \in L$ entered on a horizontal edge in $P$.  
If $L$ is noncrossing, there is a walk $P: j \to i$ only if $(i, j) \in \invq{A}$, in which case $P$ is unique and $\mathrm{wt}_{P}(\vec{c})$ is the product of $c_{b}$ weakly right of the box corresponding to $(i, j)$.  
This parametrizes a dense subset of $\mathring{X}^\lambda_{L}$ (where the $\ast$ are nonzero), so the closure $X^{\lambda}_{L}$ is the positroid variety associated to $L$.  
%
\end{proof}

Every $T$-orbit in $\QGr(r;n)$ is obtained by taking the $0,1,\ast$ chart for some $X^\lambda_{\NC}$ and replacing each $\ast$ with either a $\times$ to represent a nonzero indeterminate $\mathbb{C}^*$ or a $0$. Furthermore, it is immediate to check that the inversions in $\QGr(r;n)$ satisfy the noncrossing Le condition. We therefore conclude the following.
\begin{cor}
\label{cor:positroid varieties}
    The Le diagram of the associated positroid varieties in $\QGr(r;n)$ are precisely the noncrossing Le diagrams associated to the subset of $\mathring{X}^\lambda_{\NC}$ occupied by $\times$ as described above.
\end{cor}

\subsection{Projected and Translated Richardsons}
\label{subsec:projandtrans}

Let \(w_o \in S_n\) be the longest permutation. 
For \(u \le v\) in Bruhat order, the \emph{Richardson variety} \(X^v_u \subset \fl{n}\) is defined by
\[
X^v_u = X^v \cap w_o X^u,
\]
the intersection of a Schubert and an opposite Schubert variety. In the Grassmannian, one defines analogously the Richardson varieties \(X^\nu_\eta = X^\nu \cap w_o X^\eta\) for \(\eta \subset \nu\).

We describe two ways in which \(X^\lambda_{\NC}\) is related to Richardson varieties.
First, by \cite{KLS13, KLS14}, each positroid stratum \(X^\lambda_{\NC} \subset \Gr(r;n)\) is a \emph{projected Richardson variety} of the form
\[
X^\lambda_{\NC} = \pi\bigl(X^{w_\lambda}_u\bigr),
\]
where \(\pi|_{X^{w_\lambda}_u}\) is birational onto its image. 
Here $w_{\lambda}$ is the $r$-Grassmannian permutation $w_A$ for $A\in \binom{[n]}{r}$ corresponding to $\lambda$, introduced in Section~\ref{sec:combinatorics_gr}.
The permutation \(u\) is read from the Le diagram as a pipe dream \cite[Section 19]{Po06}: replace each box of \(L\) with an elbow tile and each box of \(\lambda \setminus L\) with a cross tile, then follow the strands from the south-east boundary to the north-west to obtain \(u\) in one-line notation. See Figure~\ref{fig:le_to_perm}, where this procedure yields \(u = 123485679\,10\).




\begin{figure}[!ht]
    \centering
    \includegraphics[scale=1]{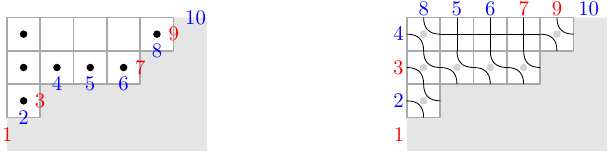}
    \caption{The Le diagram determining the permutation $u$ via a pipe dream construction}
    \label{fig:le_to_perm}
\end{figure}

\begin{cor}
    $X^\lambda_{\NC}$ is a projected Richardson variety $\pi(X^{w_{\lambda}}_u)$ for $u$ as described above.
\end{cor}

Now we give an alternate description: $X^\lambda_{\NC}$ is in fact a \emph{translated Richardson variety}. 
This follows from an analysis carried out in \cite[Section 10]{nst_c} that applies to the noncrossing Schubert cycle $X^{z_A}_{\NC}$ where $z_A\in \qgrass{r,n}$ corresponds to $\lambda$.
We record the details of the construction in the terminology of the present paper. 

Let $\alpha\in \Comp_{r,n}$ be the composition associated to $\lambda$ under the bijection of Theorem~\ref{thm:parttocomp}.
Suppose $\eta/\nu$ is the skew shape associated to the ribbon $\alpha$ (i.e. $\eta/\nu$ is a rimhook with $\eta_i-\nu_i=\alpha_i$). We choose $\eta$ and $\nu$ so that $\eta/\nu$ has its extreme boxes in the first column and first row. 
Set $w=12\cdots (r-\ell(\alpha))r(r-1)\cdots (r-\ell(\alpha)+1)(r+1)(r+2)\cdots$, and suppose further that $w_\eta$ and $w_\nu$ are the $r$-Grassmannian permutations associated with $\eta$ and $\nu$ respectively.
Then for $(x,y)=(w_{\nu}w,w_\eta w)$ we have that
$$
X^{z_A}_{\NC}=x^{-1}X^y_x.
$$ 
As discussed in \cite[Proposition 10.5]{nst_c}, since the same  $w\in S_r\times S_{n-r}$ is applied on the right to both $w_{\mu}$ and $w_{\eta}$ to obtain $x,y$, we have \[
\pi(X^y_x)=\pi(X^{w_\eta}_{w_\nu})=X^\eta_\nu.
\]
We thus obtain the following.

\begin{thm}
\label{thm:translatedRich}
With notation as before, we have the equality $X^\lambda_{\NC}=x^{-1}X^{\eta}_{\nu}$.
\end{thm}

\begin{cor}
\label{cor:translatedRich}
    Every $T$-orbit closure in $\QGr(r;n)$ is a translated Richardson variety.
\end{cor}
\begin{proof}
    This follows from Theorem~\ref{thm:translatedRich} and the fact that every $T$-orbit closure contained in a toric Richardson variety is a toric Richardson variety \cite[Proposition 8.1]{bergeron2026coxeterflagvariety}.
\end{proof}

    Figure~\ref{fig:translated_richardson} shows the relevant information for $\lambda=(4,3,1)\in \Part_{4,9}$ corresponding to  $\alpha=(2,1,3)\in \Comp_{4,9}$ under the bijection from Theorem~\ref{thm:parttocomp}. 

\begin{figure}[!ht]
    \centering
    \includegraphics[scale=1]{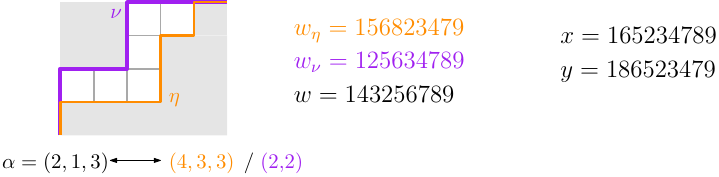}
    \caption{The data of $\eta$,  $\nu$, and $x$ determining the translated Richardson for $X_{\NC}^{(4,3,1)}$}
    \label{fig:translated_richardson}
\end{figure}

\subsection{Rigidity}

Using equivariant rigidity theorems we now show the following.
\begin{thm}
\label{thm:rigidQGr}
    Each $T$-orbit closure in $\QGr(r;n)$ is
\begin{enumerate}
    \item uniquely determined as a $T$-invariant subvariety of $\Gr(r;n)$ by its $T$-fixed point set, and
    \item containment of $T$-orbit closures in $\QGr(r;n)$ corresponds to containment of $T$-fixed points.
\end{enumerate}
\end{thm}
\begin{proof}
    To show (1) and (2) it suffices to show that if $X$ is a $T$-orbit closure in $\QGr(r;n)$ and $Y\subset \Gr(r;n)$ is a $T$-invariant subvariety with $Y^T\subset X$ then $Y\subset X$. This follows either by the fact that these $T$-orbit closures are positroid varieties (Corollary~\ref{cor:positroid varieties}) by \cite[Corollary 6.4]{BCP25}, or the fact that they are translates of Richardson varieties (Corollary~\ref{cor:translatedRich}) by \cite[Theorem 6.3]{BCP25}.
\end{proof}
These facts imply that if a face is common to two moment polytopes then the corresponding $T$-orbit closures share the same sub-$T$-orbit closure, and so the ``moment complex'' of $\QGr(r;n)$ completely determines the combinatorics of how it is assembled as a toric complex.
\section{Equivariantly quasisymmetric polynomials}
\label{sec:EQSym}
This section summarizes the essential facts about fundamental and double fundamental polynomials from \cite{BGNST1, BGNST2}. 
In what follows we let $\xl=(x_1,x_2,\ldots)$ and $\tl=(t_1,t_2,\ldots)$ be infinite variable sets, and let $\xl_k=(x_1,\ldots,x_k)$ and $\tl_k=(t_1,\ldots,t_k)$ be their finite truncations, obtained by setting $x_{k+1}=\cdots =0$ and $t_{k+1}=\cdots =0$. The adjective \emph{non-equivariant} will mean that only $\xl$ variables are used, and equivariant will mean that $\xl$ and $\tl$ variables are used. To pass from equivariant to non-equivariant we will set $\tl=0$.

The \emph{quasisymmetric polynomials} $\qsym{n}\subset \mathbb{Z}[\xl_n]$ are those polynomials which satisfy the coefficient equality of $x_1^{a_1}\cdots x_k^{a_k}$ and $x_{i_1}^{a_1}\cdots x_{i_k}^{a_k}$ for every increasing sequence $i_1<\cdots<i_k$. 
There is an equivalent condition defining quasisymmetric polynomials that we now recall. 
Define the \emph{Bergeron--Sottile map} \cite{BS98,NST_a} to be the map
$$
\rope{i}f(\xl)=f(x_1,\ldots,x_{i-1},0,x_i,x_{i+1},\ldots).
$$
In particular because of the reindexing we have $\rope{i}f(\xl_n)\in \mathbb{Z}[\xl_{n-1}]$ for $1\le i \le n$.
Then $f(\xl_n)\in \qsym{n}$ if and only if $\rope{i}f(\xl_n)=\rope{i+1}f(\xl_n)$ for $i=1,\ldots,n-1$ \cite[Theorem 2.1]{NST_a}. 

In~\cite{BGNST1} we introduced with P. Nadeau a deformation of the above equivalent definition of quasisymmetry.  Define the \emph{equivariant Bergeron--Sottile operations} are the maps
\begin{align*}\rope{i}^-f(\xl;\tl)&=f(x_1,\ldots,x_{i-1},t_i,x_{i},x_{i+1},\ldots;\tl)\\
\rope{i}^+f(\xl;\tl)&=f(x_1,\ldots,x_{i-1},x_{i},t_i,x_{i+1},\ldots;\tl).
\end{align*}
We say that $f(x_1,\ldots,x_r;\tl)$ is \emph{equivariantly quasisymmetric} if $\rope{i}^-f=\rope{i}^+f$ for $i=1,\ldots,r-1$.

\begin{defn}
Write $\eqsym{r}$ for the set of equivariantly quasisymmetric polynomials $f(\xl_r;\tl_r)$, and write $\eqsym{r}[t_{r+1},\ldots,t_n]\coloneqq \eqsym{r}\otimes \mathbb{Z}[t_{r+1},\ldots,t_n]$. 
\end{defn}

In analogy to the fundamental basis of $\qsym{r}$,~\cite{BGNST1} constructs a family of  \emph{double fundamental quasisymmetric polynomials $F_\alpha(\xl_r;\tl)$} indexed by compositions, which specialize to the fundamental polynomials recalled in the introduction when $\tl=0$.  
Here we consider only their truncations at $ t_{n+1}=t_{n+2}=\cdots=0$, which we denote by $F_\alpha(\xl_r;\tl_{n})$ (when $\alpha\in \Comp_{r,n}$ it turns out that $F_\alpha$ had no dependence on the variables outside of $\xl_r$ and $\tl_n$ before truncation).  

\begin{fact}[{\cite[Theorem 4.14]{BGNST1}}]
\label{fact:quasiFbasis}
The ring $\eqsym{r}[t_{r+1},\ldots,t_n]$ has a free $\mathbb{Z}[\tl_n]$-basis of double fundamental polynomials
$\{F_\alpha(\xl_r;\tl_n) \suchthat \ell(\alpha) \le r\}$.
\end{fact}

The second result is one of the main theorems of~\cite{BGNST1}; we restate the special case for quasigrassmannian permutations only.  
Identifying $S_{n}$ as a subset of $S_{n+1}$ in the usual way makes the $r$-quasigrassmannian permutations into a tower
\[
\qgrass{r,r+1} \subseteq  \qgrass{r,r+2}\subseteq \cdots.
\]
Let $\qgrass{r,\infty}$ to be the union of all $\qgrass{r, n}$ under this identification.  
We then have the following consequence of the ``AJS--Billey''-type formula for double forest polynomials.

\begin{fact}[{\cite[Theorem 8.3]{BGNST1}}]
\label{fact:BruhatVanishing}
Let $u \in \qgrass{r,\infty}$, and let $\alpha$ be the composition corresponding to $u$ under the bijection of ~\eqref{eq:qgrass_to_comp}, which is unique up to trailing zeros.  
Then 
\[
F_{\alpha}(t_{u(1)}, t_{u(2)}, \ldots, t_{u(r)}; \tl_{n}) 
= 
\prod_{(i,j)\in \invnc{u}} (t_j-t_i),
\]
and if $w \in \qgrass{r,\infty}$ with $w \not\ge u$ in Bruhat order, then $F_{\alpha}(t_{w(1)}, t_{w(2)}, \ldots, t_{w(r)}; \tl_{n}) = 0$. 
\end{fact}

\section{Cohomology}
\label{sec:Cohomology}

This section proves Theorems~\ref{introthm1} and~\ref{introthm4}, about the cohomology of $\QGr(r; n)$ and quasisymmetric polynomials.  
In order to do so, we first prove several results of independent interest about the equivariant cohomology ring $H_{T}^{\bullet}(\QGr(r; n))$ with respect to the standard torus $T = (\CC^{\ast})^{n}$. We identify the $T$-equivariant cohomology of a point with the polynomial ring $$H^\bullet_T(pt)=\operatorname{Sym}^\bullet \operatorname{Char}(T)=\mathbb{Z}[\tl_n],$$ where we take the standard convention used in double Schubert calculus that $t_i$ is the generator corresponding to $\chi_i^{-1}$, the inverse of the $i$'th standard character of $T$. This convention ensures for example that double Schur polynomials $s_\lambda(\xl_r;\tl_n)\in H^\bullet_T(\Gr(r;n))$ are $T$-equivariantly Kronecker dual to the $T$-equivariant homology classes of Grassmannian Schubert cycles $X^\lambda$ in $H^T_\bullet(\Gr(r;n))$.

\subsection{Graph Cohomology Rings}

Let $G$ be a finite simple graph with a labelling $\chi$ that assigns a linear form $\chi(uv) \in \ZZ[\tl_{n}]$ to each edge $uv \in E(G)$.  
The \emph{equivariant graph cohomology ring} associated to $G$ and $\chi$ is the $\ZZ[\tl_{n}]$-algebra 
\[
H^\bullet_T(G)\coloneqq \Big\{(f_v)_{v\in V(G)} \in \mathbb{Z}[\tl_{n}]^{V(G)} \;\Big|\; \chi(vw)\text{ divides } f_v-f_w\text{ for all }vw\in E(G)\Big\}.
\]
with pointwise addition, multiplication, and scaling.  We  define the \emph{singular graph cohomology ring} 
\[
H^{\bullet}(G)\coloneqq H^\bullet_T(G)\big/ (t_{1}, t_{2}, \ldots, t_{n}),
\]
where $t_{i}$ denotes the constant assignment in $H_{T}^{\bullet}(G)$ with value $t_{i}$.  

Goresky--Kottwitz--MacPherson~\cite{GKM98} show that graph cohomology rings compute the $T$-equivariant and singular cohomology rings of suitably nice $T$-spaces. 
For example, taking $G$ to be the Johnson graph $J_{r, n}$ with edge labelling 
\begin{equation}
\label{eq:edgelabels}
\chi(AB) = t_{j} - t_{i}
\qquad
\text{if $B = (A \setminus j) \cup i$ for $(i, j) \in \inv{A}$},
\end{equation}
produces $T$-equivariant and ordinary cohomology rings of the Grassmannian $\Gr(r; n)$, i.e.
$$H^\bullet_T(\Gr(r;n))\cong H^\bullet_T(J_{r,n})\qquad\text{and}\qquad H^\bullet(\Gr(r;n))\cong H^\bullet(J_{r,n}).$$

Recall from Section~\ref{sec:RecollGrass} that $V(J_{r, n})$ and $E(J_{r, n})$ can be identified with the $T$-fixed points and $T$-invariant curves in $\Gr(r; n)$.  
For $X \subseteq \Gr(r; n)$, the \emph{moment graph of $X$} is the subgraph of $J_{r, n}$ with vertex set $X^{T}$ and edges the $T$-invariant curves in $X$.  
By the proof of Theorem~\ref{thm:everythingequal}, the moment graph of $\QGr(r; n)$ is the quasisymmetric Johnson graph $QJ_{r, n}$.

\begin{thm}
\label{thm:gmk_presentation}
With the edge labellings from Equation~\eqref{eq:edgelabels},
\[
H^\bullet_T(\QGr(r;n))\cong H^\bullet_T(QJ_{r,n})\qquad\text{ and }\qquad H^\bullet(\QGr(r;n))\cong H^\bullet(QJ_{r,n}).
\]
\end{thm}
\begin{proof}
A standard application of GKM theory gives the above isomorphisms for any variety $X \subseteq \Gr(r; n)$ and its moment graph provided certain conditions hold.  
A simple statement of these conditions appears under the name``good affine paving'' in~\cite[\S11.1 items (1)--(3)]{BGNST2}, adapted from~\cite{MR2166181}.  

The affine paving of $\QGr(r;n)$ in Theorem~\ref{thm:everythingequal} is ``good'' in this sense, as each affine cell $\mathring{X}^{A}_{\NC}$ decomposes into $T$-representations with characters $\chi_{i}/\chi_{j}$ for $(i, j) \in \invq{A}$, no two characters in the same chart divide one another, and each standard character $\chi_{i}$ appears in each character with an exponent of $\pm 1$ or $0$.  
\end{proof}

\subsection{A map from quasisymmetric polynomials}

We now state an equivariant version of the first presentation in Theorem~\ref{introthm1}.  Recall the ring $\eqsym{n}$ and the polynomials $F_{\alpha}(\xl_{n}; \tl_{n})$ from Section~\ref{sec:EQSym}.  

\begin{thm}
\label{thm:QuotientPresentation}
The variety $\QGr(r; n) \subseteq \Gr(r; n)$ has
\[
H^{\bullet}_{T}(\QGr(r; n)) 
\cong \frac{\eqsym{r}[t_{r+1},\ldots,t_n]}{\langle F_\alpha(\xl_{r}; \tl_{n})\suchthat \alpha \not\in \Comp_{r,n}\rangle},
\]
and the ideal has a free $\mathbb{Z}[\tl_{n}]$-basis $\{F_{\alpha}(\xl_{r}; \tl_{n}) \suchthat \ell(\alpha)\le r,\, \alpha \notin  \mathrm{Comp}_{r, n}\}$.  
As a consequence, the first presentation of Theorem~\ref{introthm1} holds.  
\end{thm}

Theorem~\ref{thm:QuotientPresentation} is proved at the end of the subsection.  We construct an isomorphism from the right hand side of the theorem to the graph cohomology ring $H^\bullet_T(QJ_{r, n})$.  For $f(\xl_{n}; \tl_{n}) \in \ZZ[\tl_{n}][\xl_{n}]$ and $w \in S_{n}$, let
\[
f_{w} = f(t_{w(1)}, t_{w(2)}, \ldots, t_{w(n)}; \tl_{n}) \in \ZZ[\tl_{n}].  
\]
Now recall the element $z_{A} \in \qgrass{r,n} \subseteq S_{n}$ from Section~\ref{sec:QJn} associated to $A\in \binom{[n]}{r}$. In what follows when we write $\eqsym{r}\otimes \eqsym{n-r}$, the elements from the factor $\eqsym{n-r}$ are to be considered as polynomials $f(x_{r+1},\ldots,x_n;t_{r+1},\ldots,t_n)$.
\begin{lem}
\label{lem:fiber_calculation}
If $f(\xl_n;\tl_{n} ) \in \eqsym{r}\otimes \eqsym{n-r}$ and $w \in \NC_{n}^A$ then $f_{w} = f_{z_{A}}$.  
\end{lem}
\begin{proof} 
By Lemma~\ref{le:welldefined} there is a sequence $w_{1} = w, \ldots, w_{\ell} = z_{A}$ in $\NC_{n}$ such that for each $1\le k < \ell$, $w_{k+1} = w_{k} s_{i_{k}}$ for $i_{k} \in \desnc{w_{k}} \setminus \{r\}$. 
By Fact~\ref{fact:ncdescent}, each $i_{k} \in \{w_{k}(i_{k}), w_{k}(i_{k}+1)\}$, so by the definition of equivariant quasisymmetry $f_{w_{k}} = f_{w_{k+1}}$.
\end{proof}

Now define a map
\[
\begin{array}{rcl}
\Psi: \eqsym{r}\otimes \eqsym{n-r} & \to & \ZZ[\tl_{n}]^{\oplus \binom{[n]}{r}} \\
f(\xl_n;\tl_{n} )  & \mapsto & \big(f_{z_{A}} \big)_{A\in \binom{[n]}{r}}.
\end{array}
\]

\begin{prop}
\label{prop:welldefined}
The image of $\Psi$ is contained in the equivariant graph cohomology ring $H^\bullet_T(QJ_{r, n})$.
\end{prop}
\begin{proof}
We show that each $f \in \eqsym{r}\otimes \eqsym{n-r}$ satisfies the divisibility condition for each edge $AB$ with $B = (A \setminus j) \cup i$ and $(i, j) \in \invq{A}$.  
Let $w = (i\, j) z_{A}$, so that $w \cdot [r] = B$ and Lemma~\ref{lem:fiber_calculation} implies $f_{w} = f_{z_{B}}$.  
Then $f_{z_{A}} - f_{z_{B}} = f_{z_{A}} - f_{w}$, which must be divisible by $t_{j} - t_{i}$.  
\end{proof}

We now show that $\Psi$ surjects $\eqsym{r}$ onto $H^\bullet_T(QJ_{r, n})$.  
For an edge labeled graph $G$ as in the previous section, a collection of elements $\{f^{(v)} \in H^{\bullet}_{T}(G) \suchthat v \in V(G)\}$ is a \emph{flowup basis} with respect to an ordering $\prec$ on $V(G)$ \cite[Section 11.2]{BGNST2} if:
\[
\text{$f^{(v)}_{u} = 0$ whenever $v \not\preceq u$}
\qquad\text{and}\qquad
f^{(v)}_{v} = \prod_{\substack{uv \in E(G) \\ u \prec v}} \chi(uv).
\]
If a flowup basis exists, then it is a free $\mathbb{Z}[\tl_{n}]$-basis for $H^\bullet_T(G)$ \cite[Proposition 11.9]{BGNST2}.

\begin{prop}
\label{prop:flowupfundamental}
The map $\Psi$ sends the set $\{F_{\alpha}(\xl_r;\tl_n)\suchthat \alpha\in \Comp_{r,n}\}$  to a flowup basis of $H^\bullet_T(QJ_{r, n})$ with respect to the Gale order.  Further, if $\alpha \notin \Comp_{r,n}$ then $\Psi(F_{\alpha}(\xl_r;\tl_n)) = 0$.  
\end{prop}
\begin{proof}
As $A \le B$ if and only if $z_{A} \le z_{B}$ by Theorem~\ref{thm:ZaInv} and $\invq{A} = \invnc{z_{A}}$ by Theorem~\ref{thm:EdgesinQJ}, the first claim follows from Fact~\ref{fact:BruhatVanishing}.  
The second claim also follows from Fact~\ref{fact:BruhatVanishing}: if $\alpha \notin \Comp_{r,n}$ then the associated noncrossing partition $Z$ does not belong to $\NC_{n}$, and therefore cannot precede any element of $\qgrass{r,n}$ in the Bruhat order.
\end{proof}

\begin{proof}[Proof of Theorem~\ref{thm:QuotientPresentation}]
As $\eqsym{r}[t_{r+1},\ldots,t_n]$ has a free $\ZZ[\tl_{n}]$-basis of double fundamental quasisymmetric polynomials, Theorem~\ref{thm:gmk_presentation} and Proposition~\ref{prop:flowupfundamental} immediately imply that 
\[
H^{\bullet}_{T}(\QGr(r; n)) 
\cong \frac{\eqsym{r}[t_{r+1},\ldots,t_n]}{\langle F_\alpha(\xl_{r}; \tl_{n})\suchthat \alpha \not\in \Comp_{r,n}\rangle},
\]
and moreover that the ideal has a free $\mathbb{Z}[\tl_n]$-basis $\{F_{\alpha}(\xl_{r}; \tl_{n}) \suchthat \alpha \notin  \mathrm{Comp}_{r, n}\}$.  Quotienting by $(t_1,\ldots,t_n)$ then gives the first isomorphism in Theorem~\ref{introthm1}.
\end{proof}
\begin{eg}
    Figure~\ref{fig:qjr_gkm} depicts show the GKM graph for $\QGr(2;4)$ as well as examples of flowup basis elements obtained by evaluating double fundamental quasisymmetric polynomials.
    The combinatorial vine model introduced in~\cite[Section 5]{BGNST1} implies the following expansions
    \begin{align*}
        F_{21}(\xl_2;\tl_4) &= (x_{1} -t_{3})  (x_{1} - t_{1})  (x_{2} - t_{1})\\
        F_{2}(\xl_2;\tl_4)&=x_{2}^{2} + x_{1} x_{2} - x_{2} t_{3} - x_{2} t_{2} + x_{1}^{2} - x_{1} t_{3} - x_{1} t_{2} + t_{2} t_{3} - x_{2} t_{1} - x_{1} t_{1} + t_{1} t_{3} + t_{1} t_{2}
    \end{align*}
    using which the reader can check that evaluations at elements in $\qgrass{2,4}$ produces the polynomials highlighted in red.
\end{eg}

\begin{figure}[!ht]
    \centering
    \includegraphics[width=\linewidth]{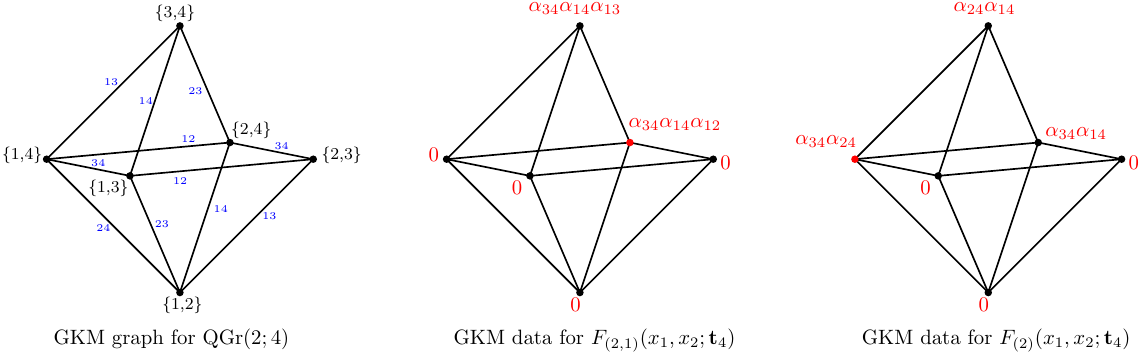}
    \caption{Examples of flowup basis elements from evaluations of $F_{\alpha}(\xl_r;\tl_n)$, where $\alpha_{ij}\coloneqq t_j-t_i$.}
    \label{fig:qjr_gkm}
\end{figure}

\subsection{A quasisymmetric Borel presentation}

We now give the second presentation in Theorem~\ref{introthm1}.  As in the previous section, we state an equivariant version which implies the ordinary version.

\begin{thm}
\label{thm:Qsymtensor}
The variety $\QGr(r; n) \subseteq \Gr(r; n)$ has
\[
H^{\bullet}_{T}(\QGr(r; n)) 
\cong 
\frac{\eqsym{r} \otimes \eqsym{n-r}}{\langle f(\xl_{n}; \tl_{n}) - f(\tl_{n}; \tl_{n}) \suchthat f \in \eqsym{n} \rangle}.
\]
As a consequence, the second presentation of Theorem~\ref{introthm1} holds.  
\end{thm}
\begin{proof}
By Proposition~\ref{prop:flowupfundamental}, composing the map $\Psi$ with the isomorphism of Theorem~\ref{thm:gmk_presentation} gives a surjection from $\eqsym{r} \otimes \eqsym{n-r}$ onto $H^{\bullet}_{T}(\QGr(r; n))$.  
We therefore only need to show that every element of $\eqsym{r}\otimes \eqsym{n-r}$ can be written modulo the ideal as a $\ZZ[\tl_{n}]$-linear combination of $\{F_{\alpha}(\xl_r;\tl_n)\suchthat \alpha\in \Comp_{r,n}\}$.  

By Fact~\ref{fact:quasiFbasis}, the top $\xl_{n}$-degree part of any element of $\eqsym{r} \otimes \eqsym{n-r}$ is an ordinary quasisymmetric polynomial.  
Thus by induction on $\xl_{n}$-degree, it is sufficient to show that every element of $\qsym{r}\otimes \qsym{n-r}$ can be written as a linear combination of the ordinary fundamentals $\{F_{\alpha}(\xl_r)\suchthat \alpha\in \Comp_{r,n}\}$ modulo the ideal $\langle f-f(0)\suchthat f\in \qsym{n}\rangle$. 
The coproduct formula for fundamental quasisymmetric polynomials $F_\beta(\xl_n)$ induces an upper triangular change of basis 
$$
\{F_{\alpha}(\xl_{r}) F_{\beta}(x_{r+1},\ldots,x_n)\suchthat \ell(\alpha)\le r, \ell(\beta)\le n-r\}\to \{F_{\alpha}(\xl_{r}) F_{\beta}(\xl_{n})\suchthat \ell(\alpha)\le r, \ell(\beta)\le n-r\}
$$
with respect to $|\beta|$. The elements $F_{\alpha}(\xl_{r}) F_{\beta}(\xl_{n})$ belong to our ideal if $|\beta| > 0$, so we need only reduce the elements $F_{\alpha}(\xl_{r})$ modulo the ideal.  

To do so, we appeal to the algorithm constructed in \cite[\S 9.2]{NST_a}; while this algorithm was originally phrased in the language of forest polynomials, we summarize the relevant case using compositions.  
Every $\alpha \notin \Comp_{r, n}$ determines two compositions $\beta = (\beta_{1}, \ldots, \beta_{k})\in \Comp_{r,n}$ and $\gamma = (\gamma_{1}, \ldots, \gamma_{\ell})\ne \emptyset$ so that $\alpha = (\gamma_{1}, \ldots, \gamma_{\ell} + \beta_{1}, \ldots, \beta_{k})$ and $|\beta|$ is as large as possible.  
Then the proofs of \cite[Theorem 9.7, Lemma 9.8]{NST_a} show that we can write 
\[
F_{\alpha}(\xl_{r}) = F_{\beta}(\xl_{r}) F_{\gamma}(\xl_{n}) + \sum_{\substack{\theta \in \Comp_{r, n} \\ |\theta| < |\beta|}} F_{\theta}(\xl_{r}) f_{\theta}(\xl_{n})
\qquad\text{with $f_{\theta}(\xl_{n}) \in \qsym{n}$ of degree $|\alpha| - |\theta|$.}
\]
As $|\alpha| - |\theta| > |\gamma| > 0$, each term on the right is contained in $\langle f-f(0)\suchthat f\in \qsym{n}\rangle$.
\end{proof}

\subsection{Kronecker duality and positivity}

We conclude the section with a proof of Theorem~\ref{introthm4}. If $
X$ is a projective (not necessarily irreducible) variety with a $T$-action, then there are \emph{Kronecker pairings}
\[
\langle -,-\rangle^T_X: H_\bullet^T(X)\otimes_{H^\bullet(pt)} H^\bullet_T(X)\to H^\bullet_T(pt)=\mathbb{Z}[\tl_n]
\]
and
\[
\langle -,-\rangle_X:H_\bullet(X)\otimes H^\bullet(X)\to H^\bullet(pt)=\mathbb{Z},
\]
where $H^\bullet_T(X)$ and $H_\bullet^T(X)$ are $T$-equivariant cohomology and $T$-equivariant Borel--Moore homology respectively.  The pairing is defined to be zero if the cohomological degree and homological dimension differ.
If $Y\subset X$ is a projective toric variety in $X$, then for cohomology classes $\wt{f}\in H^\bullet_T(X)$ and $f\in H^\bullet(X)$ we can express the Kronecker pairings $$\langle [Y],\wt{f}\rangle^T_X=\int_Y^T \wt{f}\in H^\bullet_T(pt)=\mathbb{Z}[\tl_n]\qquad\text{and}\qquad\langle [Y],\wt{f}\rangle_X=\int_Y f \in H^\bullet(pt)=\mathbb{Z},$$
where $\wt{f}$ and $f$ are considered in $H^\bullet_T(Y)$ and $H^\bullet(Y)$ via pullback from the inclusion $Y\hookrightarrow X$, and the \emph{degree maps} $\int_Y^T$ and $\int_Y$ are computed directly on $Y$ by $\kappa_{\ast}\rho^{\ast}$, where $\rho: \wt{Y}\to Y$ is any toric resolution and $\kappa: \wt{Y} \to pt$ is the projection.

Recall the affine paving $\{\mathring{X}^{A}_{\NC} \suchthat A \in \binom{[n]}{r} \}$ into $T$-representations from  Theorem~\ref{thm:everythingequal}.  
A standard property of affine pavings implies the cell closures $X^{A}_{\NC} = \overline{\mathring{X}^{A}_{\NC}}$ determine a free $\mathbb{Z}$-basis for $H_\bullet(\QGr(r;n))$~\cite[Ex.~19.1.11(b)]{Fulton_Intersection} and a free $\mathbb{Z}[\tl_n]$ basis for $H_\bullet^T(\QGr(r; n))$ \cite[Proposition 2.1]{Gra01}.
\begin{thm}
\label{thm:Kronecker}
The cohomological basis $\{F_{\alpha}(\xl_{r}; \tl_{n}) \suchthat \alpha \in \Comp_{r, n}\}$ is Kronecker dual to the homology basis $\{ [X^{A}_{\NC}]  \suchthat A \in \binom{[n]}{r} \}$, or equivalently 
\[
\langle X^A_{\NC},F_\alpha(\xl_r;\tl_n)\rangle_{\QGr(r;n)}^T=\int_{X^{A}_{\NC}}^{T} F_{\alpha}(\xl_{r}; \tl_{n}) = \begin{cases}
1 & \text{if $A = A_{\alpha}$,} \\
0 & \text{otherwise.}
\end{cases}
\]
Passing to singular cohomology, this implies Theorem~\ref{introthm4}.
\end{thm}
\begin{proof}
In~\cite[Theorem 11.11]{BGNST2} we show that the map $\mathbf{ev}_{\NC}: H^\bullet_T(\hhmp_n) \cong \mathbb{Z}[\tl_n][\xl_n]/\langle f(\xl_n;\tl_n)-f(\tl_n;\tl_n)\suchthat f\in \eqsym{n}\rangle \to \mathbb{Z}[\tl_n]^{\oplus \NC_n}$ given by $f(\xl_{n}; \tl_{n}) \mapsto (f_{w})_{w \in \NC_{n}}$ is an injection. 
Then by Theorem~\ref{thm:gmk_presentation} we have a commutative diagram
\begin{center}
    \begin{tikzcd}
        H^\bullet_T(X^{z_A}_{\NC})&
        \ar[l]H^\bullet_T(\hhmp_n)\ar{r}{\mathbf{ev}_{\NC}}& 
        \mathbb{Z}[\tl_n]^{\oplus \NC_n}\\
        H^\bullet_T(X^{A}_{\NC})\ar{u}{\pi^{\ast}}&
        \ar[l]H^\bullet_T(\QGr(r;n))\arrow{r}{\Psi} \ar{u}{\pi^{\ast}}&
        \mathbb{Z}[\tl_n]^{\oplus \binom{[n]}{r}} \ar[u]
    \end{tikzcd}
\end{center}
where the unlabeled horizontal arrows are pullbacks of inclusions and the unlabeled vertical arrow sends $(f_{A})_{A \in \binom{[n]}{r}}$ to $(f_{w \cdot [r]})_{w \in \NC_{n}}$.  

Now we claim that the middle vertical arrow sends $f(\xl_r;\tl_n)\in \eqsym{r}$ to $f(\xl_r;\tl_n) \in H^\bullet_T(\hhmp_n)$. 
By Lemma~\ref{lem:fiber_calculation}, $f_{w} = f_{z_{A}}$ for $A = w \cdot [r]$, so $\mathbf{ev}_{\NC}(f)$ is the image of $\Psi(f)$ under the right vertical arrow.
The right three maps are injective so this determines the middle arrow by commutativity. 

By Theorem~\ref{thm:posfibers} the map $\pi:X^{z_A}_{\NC}\to X^A_{\NC}$ is $T$-equivariant and birational.  
Thus by the push-pull formula and the commutativity of the left part of the diagram, we have
\[
\int^T_{X^{A}_{\NC}}F_\alpha(\xl_r;\tl_n)=\int^T_{X^{z_A}_{\NC}}F_\alpha(\xl_r;\tl_n),
\]
where the $F_\alpha(\xl_r;\tl_n)$ on the right is pulled back from $H^\bullet_T(\hhmp_n)$. By \cite[Theorem 12.12]{BGNST2} this is the indicator function for whether $z_A$ corresponds to $\alpha$ under the bijection $\qgrass{r,n}\to \Comp_{r,n}$.
\end{proof}

We conclude by considering the degree operations of other cohomology classes.  
Recall that $f \in \mathbb{Z}[t_1,\ldots,t_n]$ is \emph{Graham-positive} if $f\in \mathbb{N}[t_2-t_1,\ldots,t_n-t_{n-1}]$.\footnote{Each $t_{i+1}-t_i$ corresponds to a positive root $\chi_i/\chi_{i+1}$ under our convention that $t_i$ corresponds to $\chi_i^{-1}$.}  
The $T$-equivariant homology class of every $T$-invariant subvariety $Z \subseteq \Gr(r; n)$ can be written as 
\[
[Z] = \sum_{\eta \in \Part_{r, n}} a_{\eta}(\tl_{n})[X^\eta]
\qquad\text{with}\qquad
a_{\eta}(\tl_{n}) =\langle [Z], s_{\eta}(\xl_r; \tl_{n})\rangle^T_{\Gr(r;n)}, 
\]
where $s_{\eta}(\xl_r; \tl_{n})$ is the double Schur polynomial.
Graham showed \cite{Gra01} that for any $Z$, the $a_{\eta}(\tl_{n})$ above are Graham-positive.  
For $Z=X^\lambda_\mu$, the coefficients are the Littlewood--Richardson coefficients $c^\lambda_{\mu,\eta}(\tl_n)$ for double Schur polynomial multiplication~\cite{KT03}, which recover the usual Littlewood--Richardson after passing to singular cohomology.  

\begin{thm}
For any $T$-orbit closure $Z\subset \QGr(r;n)$, the Graham-positive quantity
\[
\int_{Z}^{T} s_\lambda(\xl_r;\tl_n)=\langle [Z],s_\lambda(\xl_r;\tl_n)\rangle^T_{\Gr(r;n)}\in \mathbb{N}[t_2-t_1,\ldots,t_n-t_{n-1}]
\]
can be computed with a manifestly positive combinatorial rule.  
\end{thm}
\begin{proof}
First, we claim that $Z=\overline{T\cdot x}\subset \QGr(r;n)$ is the birational projection of a $T$-orbit closure $Z'=\overline{T\cdot y}\subset \hhmp_n$. Indeed, $\QGr(r;n)=\bigsqcup \mathring{X}^A_{\NC}$ (Theorem~\ref{thm:everythingequal}) and the projection $\pi$ restricts to an isomorphism $\mathring{X}^{z_A}_{\NC}\cong \mathring{X}^A_{\NC}$ (Theorem~\ref{thm:posfibers}), so if $x\in \mathring{X}^A_{\NC}$ we can take the preimage $y\in \mathring{X}^{z_A}_{\NC}$. The push-pull formula shows $\int_{Z}^Ts_\lambda(\xl_r;\tl_n)=\int_{Z'}^Ts_\lambda(\xl_r;\tl_n)$, and combinatorially positive rules for degrees of $s_\lambda(\xl_r;\tl_n)$ on $T$-orbit closures in $\hhmp_n$ were computed in \cite[Corollary 12.14]{BGNST2}.
\end{proof}

Applied to each $X^\lambda_{\NC}$, this gives that 
\[
\int^T_{X^\lambda_{\NC}}s_\eta(\xl_r;\tl_n)=\langle [X^\lambda_{\NC}],s_\lambda(\xl_r;\tl_n)\rangle^T_{\Gr(r;n)}=[F_\alpha(\xl_r;\tl_n)]s_\eta(\xl_r;\tl_n)
\]
is Graham-positive, a special case of the double Schubert into double forest Graham-positivity \cite{BGNST1}.

\section{Forests, Moment polytopes, and equidimensionality}
\label{sec:morefacts}

We now describe the moment polytope of each cell \(X^\lambda_{\NC}\) and prove that the moment complex of \(\QGr(r;n)\) is equidimensional. 
We do so by drawing on similar results for $\hhmp_{n}$  from~\cite{BGNST1, BGNST2}. 

\subsection{Forests and noncrossing partitions}

We first recall some essential combinatorics about indexed forests.  
A \emph{plane binary tree} is a tree \(T\) in which each node is either a leaf or an internal node \(v \in \internal{T}\) with one left child \(v_L\) and one right child \(v_R\). The size of \(T\) is \(|T| \coloneqq |\internal{T}|\). There is a unique tree \(\ast\) of size \(0\), consisting of a single node.

A \emph{plane binary indexed forest} is a sequence \(F = (T_1, T_2, \dots)\) of plane binary trees such that all but finitely many \(T_i\) are equal to \(\ast\). Let \(\indexedforests\) denote the set of all such forests, and let \(\indexedforests_n\) denote the subset consisting of those indexed forests whose nontrivial components have leaf labels in \([n]\). 
There is a bijection
\[
\ForToNC : \indexedforests_n \longrightarrow \NC_n
\]
constructed as follows: for each internal node \(v \in \internal{F}\) that is a right child, delete the edge joining \(v\) to its left child \(v_L\). Consider the resulting connected components. To each one assign a cycle written in decreasing order on its set of leaf labels.

For \(z \in \qgrass{r,n}\), the associated tree \(T \in \Zigzag_{r,n}\) is called a \emph{zigzag tree}. These trees are in bijection with compositions in a manner compatible with the bijection $z_{A} \mapsto \alpha_{A}$, see Figure~\ref{fig:zigzags_comps}.
We say that $T\in \Zigzag_{r,n}$ is \emph{maximal} if $|T|=n-1$.

\begin{figure}[!ht]
    \centering
    \includegraphics[scale=0.8]{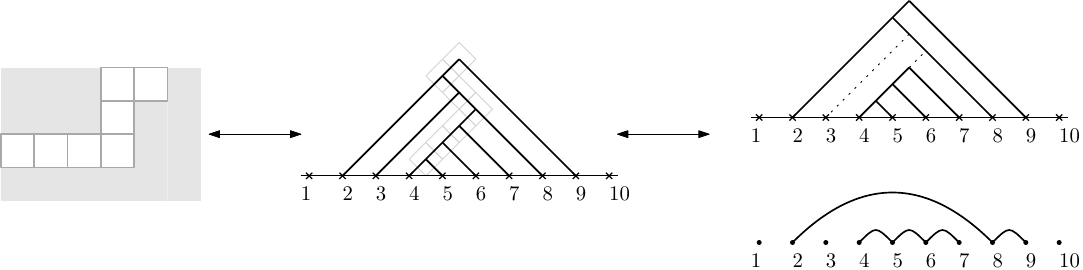}
    \caption{Mapping $(2,1,4)\in \Comp_{4,10}$ to a zigzag tree in $\Zigzag_{4,10}$ and then to the associated permutation in $\qgrass{4,10}$.}
    \label{fig:zigzags_comps}
\end{figure}

\subsection{Moment Polytopes}

We describe the moment polytopes of the $X^\lambda_{\NC}$ in terms of hyperplanes. 

\begin{thm}
Let $A=(a_k,\dots,a_1\mid_r b_1,\dots,b_k)\in \binom{[n]}{r}$.
The moment polytope of $X^{\lambda_A}_{\NC}$ is the subset of $[0,1]^n\cap \{z_1+\cdots+z_n=r\}$ cut out by:
\begin{enumerate}[label=\normalfont(\roman*)]
    \item the equalities $z_i=1$ for $1\leq i<a_k$ and $z_i=0$ for $b_k<i\leq n$;
    \item $z_{a_i}+z_{a_i+1}+\cdots+z_{b_i}\leq r+1-a_i$ \quad for $1\leq i\leq k$;
    \item $z_q+z_{q+1}+\cdots+z_{b_{i-1}} \geq r-q$ \quad for $q\in [r]$ with $a_i\leq q<a_{i-1}$,
\end{enumerate}
where we set $a_0=r+1$ and $b_0=r$ by convention.
\end{thm}
\begin{proof}
    This follows by specializing \cite[Theorem 7.6]{BGNST2} which gives an explicit description of the moment polytope of $X^u_{\NC}\subset \fl{n}$ for $u\in \NC_n$ as a polypositroid \cite{LP20}. 
    More precisely we set the regular dominant weight in the description in \cite[Theorem 7.6]{BGNST2} to the vector $(1^r,0^{n-r})$.
\end{proof}

We can also give an explicit vertex description for the moment polytopes once we set up some more notation. 
It will be convenient for us to use compositions as our indexing objects.
Let $\alpha\in \Comp_{r,n}$. 
As in Theorem~\ref{thm:parttocomp},
parse $\alpha$ uniquely as $d_11^{e_1-1}(d_2+1)1^{e_2-1}\cdots (d_k+1)1^{e_k-1}$ subject to the constraints $d_i,e_i\geq 1$ for $1\leq i\leq k$.
This datum determines a decomposition of the interval $\{r-\ell(\alpha)+1,\dots,r-\ell(\alpha)+1+|\alpha|\}$ into disjoint intervals from left to right:
\begin{align}\label{eq:e_d_decomposition}
   \{r-\ell(\alpha)+1,\dots,r-\ell(\alpha)+1+|\alpha|\} =\left(E_1\sqcup \cdots \sqcup E_k\right) \sqcup \left(D_k\sqcup \cdots \sqcup D_1\right),
\end{align}
where $|E_i|=e_i, |D_i|=d_i$ for $1\leq i\leq k$, so that \[
\{r-\ell(\alpha)+1,\dots,r\}=E_1\sqcup \cdots E_k \text{ and } \{r+1,\dots,r+1-\ell(\alpha)+|\alpha|\}=D_k\sqcup \cdots \sqcup D_1.\]


\begin{defn}\label{def:admissible}
    Fix $\alpha\in \Comp_{r,n}$.
    Let $A=(a_{m},\dots,a_1\mid_r b_1,\dots,b_{m}) \in \binom{[n]}{r}$ where $r-\ell(\alpha)+1\leq a_{m}$ and $b_m\leq r-\ell(\alpha)+1+|\alpha|$. 
    Then the data $E_{1}, \ldots, E_{k}, D_{k}, \ldots, D_{1}$ as in~\eqref{eq:e_d_decomposition} determines unique indices $(i_1,j_1)$, $(i_2,j_2),\dots,(i_{m},j_{m})$ such that $a_p\in E_{i_p}$ and $b_p\in D_{j_p}$ for all $1\leq p\leq m$.
    We say that $A$ is \emph{$\alpha$-admissible} if:
    \[
        j_1\leq i_1 < j_2\leq i_2 < \cdots < j_{m} \leq i_{m}.
    \]
\end{defn}
Figure~\ref{fig:admissible_inadmissible} gives an example of an $\alpha$-admissible set $A$ and an $\alpha$-inadmissible set $B$ for $\alpha=(2,1,4)\in \Comp_{4,10}$.
For $B=(3,4\mid_4 5,6)$ we have, in terms of the notation in Definition~\ref{def:admissible}, that $(i_1,j_1)=(1,2)$ and $(i_2,j_2)=(2,2)$. 
Since $j_1>i_1$ we have $B$ is $\alpha$-inadmissible.

\begin{figure}[!ht]
    \centering
    \includegraphics[scale=0.8]{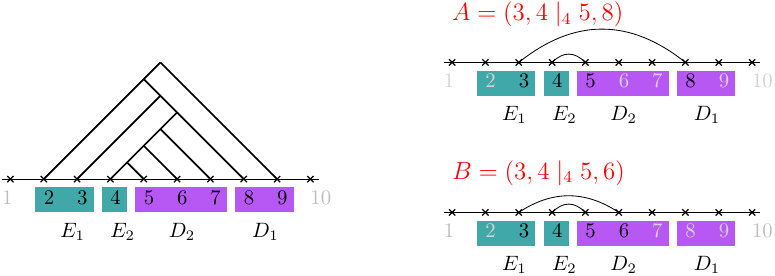}
    \caption{For $\alpha=(2,1,4)\in \Comp_{4,10}$ we have $\alpha$-admissible $A$ and $\alpha$-inadmissible $B$ . The arcs connect $i/j$-indices determined as in Definition~\ref{def:admissible}.}
    \label{fig:admissible_inadmissible}
\end{figure}
\begin{thm}
    Let $\alpha\in \Comp_{r,n}$.
    The $T$-fixed points of $X_{\NC}^{\lambda_{\alpha}}$ are
    \[
       (X^{\lambda_{\alpha}}_{\NC})^{T}=\left\{A\in {\textstyle \binom{[n]}{r}}\;\middle|\; A \text{ is  $\alpha$-admissible}\right\}.
    \]
\end{thm}
\begin{proof}
    We outline the argument, referring the reader to~\cite{BGNST1} for any errant terminology about indexed forests. 
    Let $Z\in \Zigzag_{r,n}$ correspond to $\alpha$.
    For $v\in \internal{Z}$, let $\tau_v$ be the transposition $(i\,j)$ where $i$ is the rightmost leaf descended from $v_L$ and $j$ is the rightmost leaf descended from $v$. 
    Order the internal nodes $v_1,v_2,\ldots \in \internal{F}$ in the unique order from root to terminal node. Then~\cite[Theorem 6.3]{BGNST2}
    implies that the set of $T$-fixed points of $X_{\NC}^{\lambda_{\alpha}}$ is given by the following subset of $\binom{[n]}{r}$:  
    \begin{equation}
    \label{eq:vertices}\{\tau_{v_{i_1}}\tau_{v_{i_2}}\cdots \tau_{v_{i_k}}\cdot [r]\suchthat 1\le i_1 < i_2<\cdots < i_k \le n-1\}.
    \end{equation}
    Concretely, the products $\tau_{v_{i_1}}\tau_{v_{i_2}}\cdots \tau_{v_{i_k}}$ correspond to deleting all left edges emanating from internal nodes in $\internal{Z}\setminus \{v_{i_1},\dots,v_{i_k}\}$, and then multiplying the backward cycles on the leaves for each connected component. This perspective, in conjunction with Lemma~\ref{lem:DescribeNCToParts}, implies the claim. 
\end{proof}

For $\alpha=(2,1)\in \Comp_{2,4}$ we have five $T$-fixed points corresponding to the elements in $\binom{[4]}{2}$ except $\{3,4\}$ which is $\alpha$-inadmissible. 
These in turn give the vertices of the blue pyramid in Figure~\ref{fig:placeholder}.
Figure~\ref{fig:all_relevant_subproducts_zigzag} show how these admissible sets arise following the left edge deletion procedure described in the preceding proof.

\begin{figure}[!ht]
    \centering
    \includegraphics[scale=1]{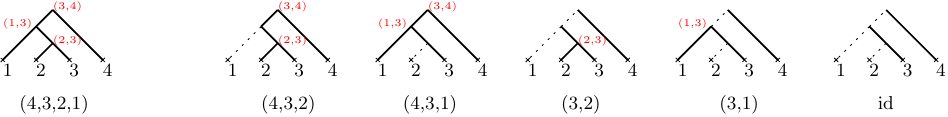}
    \caption{$Z\in \Zigzag_{2,4}$ for  $\alpha=(2,1)$ and the subproducts giving distinct $T$-fixed points}
    \label{fig:all_relevant_subproducts_zigzag}
\end{figure}




\begin{cor}
$\QGr(r;n)$ is an equidimensional complex of $(n-1)$-dimensional toric varieties.
\end{cor}
\begin{proof}
Let $T_{A}$ and $T_{B}$ be zigzag trees associated to sets $A, B \in \binom{[n]}{r}$. By \eqref{eq:vertices}, if $T_B$ is obtained by iteratively deleting the root vertex of $T_A$ (i.e. it is a lower ideal of the tree $T_A$), then $(X^B_{\NC})^T\subset (X^A_{\NC})^T$, which then implies $X^B_{\NC}\subset X^A_{\NC}$ by Theorem~\ref{thm:rigidQGr}. As $\dim X^{A}_{\NC}=|T_A|$, equidimensionality follows because we can extend any zigzag tree to a maximal one, possibly in many ways.  If $B$ corresponds to the composition $\alpha=(\alpha_1,\ldots,\alpha_k)$ then one uniform choice is to have $A$ correspond to $(a,1^b,\alpha_1,\ldots,\alpha_k)$ with $b=r-k-1$ and $a=n-1-b-|\alpha|$. 
\end{proof}

\bibliographystyle{hplain}
\bibliography{main.bib}
\end{document}